\documentclass[12pt, reqno]{amsart}

\usepackage{amsmath}
\usepackage{mathtools}
  \usepackage{paralist}
  \usepackage{graphics} 
  \usepackage{epsfig} 
\usepackage{graphicx}  \usepackage{epstopdf}
 \usepackage[colorlinks=true]{hyperref}
\hypersetup{urlcolor=blue, citecolor=red}

\numberwithin{equation}{section}

\theoremstyle{dgthm}
\newtheorem{thm}{Theorem}[section]
\newtheorem{theorem}[thm]{Theorem}
\newtheorem{lemma}[thm]{Lemma}

\newtheorem{proposition}[thm]{Proposition}
\theoremstyle{dgdef}
\newtheorem{remark}[thm]{Remark}
\newtheorem{definition}[thm]{Definition}

\DeclarePairedDelimiter{\norm}{\lVert}{\rVert}
\DeclarePairedDelimiter{\sprod}{\langle}{\rangle}

\newcommand{\diff}{\, \mathrm{d}}

\usepackage{mathtools}
\mathtoolsset{showonlyrefs}

\title[Brake Orbits and Finsler Geodesics]
{Brake orbits for Hamiltonian systems of classical type via Finsler geodesics}

\author[D. Corona, F. Giannoni]{}


\email{dario.corona@unicam.it,fabio.giannoni@unicam.it}

\begin{document}
\maketitle

\centerline{\scshape Dario Corona$^*$, Fabio Giannoni$^*$}
\medskip
{\footnotesize
	\centerline{School of Science and Technology}
	\centerline{Mathematics Division, University of Camerino}
\centerline{Camerino, Italy} 
} 


\begin{abstract}{
We consider Hamiltonian functions of classical type, namely 
even and convex with respect to the generalized momenta.
A brake orbit is a periodic solution of Hamilton's equations
such that the generalized momenta are zero on two different points.
Under mild assumptions,
this paper reduces the multiplicity problem of the brake orbits for a 
Hamiltonian function of classical type 
to the multiplicity problem of orthogonal geodesic chords
in a concave Finslerian manifold with boundary.
This paper will be used for a generalization of a Seifert's conjecture
about the multiplicity of brake orbits
to Hamiltonian functions of classical type.
}
\end{abstract}

\section{Introduction}%
\label{sec:introduction}

Let $H\colon \mathbb{R}^{2n} \to \mathbb{R}$ be an autonomous Hamiltonian function of class $C^2$.
A curve $(q,p)\colon [0,T] \to  \mathbb{R}^{2n} $ is a solution of Hamilton's equations if
\begin{equation}
\label{eq:hamiltonianSys}
\dot{q} = \frac{\partial H}{\partial p}(q,p),
\qquad\text{and}\qquad
\dot{p} = -\frac{\partial H}{\partial q}(q,p).
\end{equation}
Since the Hamiltonian is autonomous,
the conservation law of the energy holds.
More formally,
if $(q,p)\colon [0,T] \to \mathbb{R}^{2n}$ is a solution of Hamilton's equations, then
there exists a real number $E$, called energy, such that
\begin{equation*}
H(q(t),p(t)) = E, \quad \forall t \in [0,T].
\end{equation*}
We give the following definitions (cf. \cite{Weinstein1978}).
\begin{definition}
	\label{def:Hclassical}
	A function $H(q,p)$ on $\mathbb{R}^{2n} $ is of classical type if, for each $q_0 \in \mathbb{R}^{n} $, 
	the function $p \mapsto H(q_0,p)$ is even and
	$
	(\partial^2 H/\partial p^2)(q_0,p)
	$
	is strictly positive definite for all $p$,
	namely there exists a continuous function $\nu\colon \mathbb{R}^{n} \to \mathbb{R}$
	such that, for all $q \in \mathbb{R}^{n} $, $\nu(q)> 0$ and
	\begin{equation}
	\label{eq:defPosDef}
	\frac{\partial^2 H}{\partial p^2}(q,p)[\xi,\xi]
	\ge \nu(q)\norm{\xi}^2,
	\qquad \forall p,\xi \in \mathbb{R}^{n}.
	\end{equation}
\end{definition}

\begin{remark}
	\label{rmk:notation-solutions}
	If $H$ is a Hamiltonian of classical type,
	by \eqref{eq:defPosDef} the inverse of 
	$(\partial H/\partial p)(q,\cdot)$ is well defined for all $q \in \overline{D}$.
	Hence, with a slight abuse of notation,
	we will say that a curve $q\colon [0,T] \to \overline{D}$
	is a solution of the Hamilton's equations if 
	$(q,p)\colon [0,T] \to \mathbb{R}^{2n}$ is a solution of \eqref{eq:hamiltonianSys},
	where $p$ is implicitly defined by
	\[
	\dot{q}(t) = \frac{\partial H}{\partial p}(q(t),p(t)),
	\quad \forall t \in [0,T].
	\]
\end{remark}

\begin{definition}
	\label{def:V&K}
	Let $H\colon \mathbb{R}^{2n} \to \mathbb{R}$ be a Hamiltonian function of classical type.
	We define the potential energy function
	$V\colon \mathbb{R}^{n} \to \mathbb{R}$ as 
	\[
	V(q) = H(q,0), \qquad \forall q \in \mathbb{R}^{n},
	\]
	and the kinetic energy function $K\colon \mathbb{R}^{2n} \to \mathbb{R}$ as
	\[
	K(q,p) = H(q,p) - V(q).
	\]
\end{definition}
By Definitions \ref{def:Hclassical} and \ref{def:V&K},
a Hamiltonian of classical type can be written as
\[
H(q,p) = K(q,p) + V(q),
\]
where $K(q,p)$ 
is even with respect to $p$,
strictly positive unless $p = 0$ and, for each $q \in \mathbb{R}^{n}$, 
\[
\frac{\partial^2 K}{\partial p^2}(q,p)[\xi,\xi]
\ge \nu(q)\norm{\xi}^2,
\qquad \forall p,\xi \in \mathbb{R}^{n}.
\]

\begin{definition}
	Let $H$ be a Hamiltonian function of classical type.
	A potential well for $H$ is an open set $D \subset \mathbb{R}^{n}$
	with boundary $\partial D$ of class $C^2$ such that,
	for some real number $E$, the followings hold:
	\begin{itemize}
		\item $V(q) < E$ on $D$;
		\item $V(q) = E$ on $\partial D$;
		\item $\nabla V(q) \ne 0$, for all $q \in \partial D$.
	\end{itemize}
\end{definition}

\begin{definition}
	Let $D \subset \mathbb{R}^{n}$ be a potential well for a Hamiltonian $H$, 
	with $V(q) = E$ on $\partial D$.
	A solution $(q(t),p(t))$ of Hamilton's equations for $H$ is called brake orbit
	if it has energy $E$ and 
	there exists $T > 0$ such that $q(t) \in D$ for $0 < t < T$, while $q(0),q(T)\in \partial D$.
\end{definition}
Following the notation of Remark \ref{rmk:notation-solutions},
we say that $q\colon [0,T] \to \overline{D}$ is a brake orbit if 
it is a solution of \eqref{eq:hamiltonianSys} with energy $E$,
$q(]0,T[) \subset D$ and $q(0),q(T) \in \partial D$.
\begin{remark}
	By the conservation law of the energy, if $(q(t),p(t)) $ is a brake orbit,
	then $p(0)$ and $p(T)$ must be zero.
	Since $H$ is of classical type (hence even in $p$),
	the solution can be continued so that it will be periodic.
	In other words, $q(t)$ oscillates back and forth along a curve in $D$ with endpoints in $\partial D$.
\end{remark}

This paper concerns 
the multiplicity of the brake orbits in a bounded potential well.
When the Hamiltonian is natural, hence given by
\begin{equation*}
H(q,p) = \frac{1}{2}\sum_{i,j = 1}^n a^{ij}(q) p_i p_j + V(q),
\end{equation*}
where $(a^{ij}(q))$ is a positive definite quadratic form on $\mathbb{R}^n$,
and the potential well is homeomorphic to the $n$-dimensional disk in $\mathbb{R}^n$,
H. Seifert conjectured the existence of at least $n$ brake orbits
(cf. \cite{Seifert1948}).
This conjecture has motivated an extensive literature on the subject
(e.g.
\cite{Ambrosetti1993,
	hu2020,
	Liu2021,
	Liu2014,
	Long2006,
	Rabinowitz1995,
	szulkin1989,
	Wang2020,
	Zhang2014})
and it has been recently proved in \cite{Giambo2020Seifert},
exploiting some partial results given by the authors in different previous papers
(cf. \cite{Giambo2009,Giambo2010,Giambo2015,giambo2015-nonLinAnal,Giambo2018,Giambo2004}).
This work points towards a generalization of Seifert's conjecture,
looking for the multiplicity of brake orbits when the Hamiltonian function is of classical type.
Indeed,
the present paper includes 
some results that will be exploited in the future 
to generalize the Seifert's conjecture for Hamiltonian systems of classical type.
Different generalizations of Seifert's conjecture have been analysed in the
last decades.
The papers with the most similar setting to the present one are
\cite{ruiz1977} and \cite{Weinstein1978},
where the existence of one brake orbit is proved
for Finsler mechanical systems and Hamiltonian systems of classical type,
respectively.

Let us introduce the following notation.
Let $(\mathcal{M},F)$ be a Finsler manifold of class $C^3$
and let ${\Omega}\subset \mathcal{M}$ be an open subset with boundary $\partial\Omega \in C^2$
(we refer to \cite{rund1959,Shen2001} for a background material about Finsler geometry).
\begin{definition}
	A curve $\gamma\colon [a,b] \to \overline\Omega$ is a Finsler geodesic chord if 
	\begin{itemize}
		\item it is a geodesic with respect to the Finsler metric $F$;
		\item $\gamma(a),\gamma(b) \in \partial\Omega$ and
		$\gamma(]a,b[) \subset \Omega$.
	\end{itemize}
	If $\dot{\gamma}(a)$ and $\dot{\gamma}(b)$ are orthogonal,
	with respect to the Finsler metric $F$, to
	$T_{\gamma(a)}\partial\Omega$ and $T_{\gamma(b)}\partial\Omega$ respectively,
	namely
	\begin{equation}
	\label{eq:conBounCond}
	\frac{\partial F^2}{\partial v} (\gamma(t),\dot{\gamma}(t))[\xi] 
	\coloneqq
	\frac{\diff}{\diff s}
	F^2(\gamma(t), \dot{\gamma}(t) + s\xi)\bigg|_{s = 0}
	= 0,
	\end{equation}
	for all $ \xi \in T_{\gamma(t)}\partial\Omega$,
	with $t = a,b$, 
	then $\gamma$ is called orthogonal Finsler geodesic chord.
\end{definition}

This paper reduces the multiplicity problem of the brake orbits
in a bounded potential well of a Hamiltonian of classical type
to the related problem of orthogonal geodesic chords
in a Finslerian manifold with smooth boundary.

The last ingredient to state our main theorem is 
the notion of concavity of a Finsler manifold with boundary
(cf. \cite{Bartolo2011} for the notion of convexity).
We say that $\overline{\Omega}$ is concave with respect to the Finsler metric $F$ 
if every geodesic which is tangent to $\partial\Omega$ on one point $q$
lies inside $\Omega$ on a neighbourhood of $q$.
More formally,
since $\overline{\Omega}$ is of class $C^2$, there exists a function 
$\psi\colon \mathcal{M} \to \mathbb{R}$ of class $C^2$ such that
$\psi(\Omega) \subset ]0,\infty[$,
$\psi(\partial\Omega) = 0$ and
$(\partial\psi/\partial q )(q)\ne 0$ for all $q \in \partial\Omega$.

\begin{definition}
	\label{def:Finsler-concavity}
	The manifold with boundary $\overline{\Omega}$ is strongly concave if 
	for all $q \in \partial\Omega$ we have
	\begin{equation}
	\label{eq:def-Concave}
	H_{\psi}(q,v)[v,v]
	\coloneqq \frac{\diff^2}{\diff s^2}(\psi\circ\gamma)(0)
	> 0,
	\quad \forall v \in T_q\partial\Omega,
	\ v \ne 0,
	\end{equation}
	where $\gamma\colon (-\epsilon,\epsilon) \to \mathcal{M}$
	is the unique geodesic such that $\gamma(0) = q$ and $\dot{\gamma}(0) = v$.
\end{definition}

Now we are ready to state our main theorem.
\begin{theorem}
	\label{teo:main}
	Let $D \subset \mathbb{R}^{n}$ be a potential well for a Hamiltonian $H$ of classical type.
	If $\overline{D}$ is compact,
	there exists a open set ${\Omega}\subset D$,
	with a Finsler metric $F$ on $\overline{\Omega}$, such that the following statements hold:
	\begin{itemize}
		\item $\overline{\Omega}\subset  D$;
		\item $\partial \Omega$ is of class $C^2$;
		\item $\overline{\Omega}$ is homeomorphic to $\overline{D} $;
		\item $\overline{\Omega} $ is strongly concave with respect to the Finsler metric $F$;
		\item if $\gamma\colon [0,1]\to \overline{\Omega}$ is an orthogonal Finsler geodesic chord,
		then there exists $[\alpha,\beta]\supset [0,1] $ and
		an unique continuous extension $\hat{\gamma}\colon [\alpha,\beta]\to \overline{D}$ of $\gamma$
		such that
		\begin{itemize}
			\item $\hat{\gamma}$ is a geodesic in $]\alpha,\beta[$;
			\item up to a time reparametrization,
			$\hat\gamma\colon [\alpha,\beta] \to \overline{D}$ is a brake orbit,
			namely there exists a diffeomorphism
			$\sigma\colon [0,T] \to [\alpha,\beta]$
			such that $q = \gamma \circ \sigma\colon [0,T] \to \overline{D}$
			is a brake orbit.
		\end{itemize}
	\end{itemize}
\end{theorem}

Theorem \ref{teo:main} reduces the study of multiple brake orbits 
of a Hamiltonian of classical type
to the study of multiple orthogonal geodesic chords
in a strongly concave Finsler manifold with boundary.
Given a bounded potential well $D$, we will construct a Finsler manifold
$(\overline{\Omega},F)$, with $\overline{\Omega}\subset D$, such that
there exists a bijection between the brake orbits in $D$ 
and the orthogonal geodesic chords in $(\overline{\Omega},F)$.
This result generalizes the one presented in \cite{Giambo2004}
for natural Hamiltonian functions,
where the Finsler metric is actually a Riemannian one.

Some results about the multiplicity of orthogonal geodesic chords
in the case of convex Finsler manifolds with boundary
and some generalizations can be found, 
for instance,
in \cite{Corona2020JFPT,corona2021DCDS,Corona2020Symmetry}.

\subsection{Notation}
If $f$ is a real-valued function defined on $\mathbb{R}^{2n}$,
then $\partial f/\partial q$ and $\partial f/\partial p$ will denote the differentials of $f$ with respect to $q$ and $p$ respectively.
We denote by $f'$ the differential of $f$, hence $f'(q,p) = (\partial f/\partial q, \partial f/\partial p)$.
We will denote by $v$ the conjugate variable of $p$ via Legendre transform of a function.
Hence, $\partial f/\partial v$ will denote the partial derivative with respect to $v$.
We denote by $\sprod{\cdot,\cdot}\colon \mathbb{R}^{2n}\to\mathbb{R}$ the euclidean scalar product
and $\norm{\cdot}\colon \mathbb{R}^{2n} \to \mathbb{R}$ the euclidean norm.
The euclidean distance function between two points $q_1$ and $q_2
$ will be denoted by $\text{dist}(q_1,q_2)$.
We denote by $J \in M_{2n \times 2n}(\mathbb{R})$ the symplectic matrix
\[
J = 
\begin{pmatrix}
0 & I_{n \times n} \\ -I_{n \times n} & 0
\end{pmatrix}
.
\]
Let $z\colon [0,T] \to \mathbb{R}^{2n}$ be a curve with
$z(t) = (q(t), p(t))$.
Using this notation, \eqref{eq:hamiltonianSys} can be written as
\[
\dot{z}(t) = J H'(z(t)).	
\]
For every compact interval $I \subset \mathbb{R}$
and every $A \subset \mathbb{R}^n$,
we denote by 
$W^{1,2}(I,A)$ the Sobolev space
\[
W^{1,2}(I,A) = 
\left\{
\gamma\colon I \to A:	
\gamma \text{ is absolutely continuous and } \dot{\gamma} \in L^2(I,\mathbb{R}^n)
\right\}.
\]

\section{The Jacobi-Finsler metric}
\label{sec:JF-metric}
Let $H$ be a Hamiltonian of classical type and
$D\subset \mathbb{R}^{n} $ be the (open) potential well such that $V(q)\equiv E$ on $\partial D$ and $\overline{D}$ is compact.
In this section, following the same construction of \cite{Weinstein1978},
we endow the potential well with a Finsler metric whose geodesics are linked to the solution of the Hamiltonian system via time reparametrization.
Let us define
\begin{equation*}
\Sigma = \left\{(q,p) \in \mathbb{R}^{2n}: q \in D, H(q,p) = E\right\}.
\end{equation*}
If $(q_0,p_0) \in \Sigma$, then $p_0 \ne 0$ and this implies that $H'(q_0,p_0)$ is different from zero.
As a consequence, $\Sigma$ is a regular level surface for $H$.
\begin{lemma}
	\label{lem:fromH-toU}
	There exists a function $U\colon D\times \mathbb{R}^n \to \mathbb{R}$ such that
	\begin{itemize}
		\item $U$ is of class $C^1$;
		\item $U$ is of class $C^2$ on $D \times (\mathbb{R}^n\backslash\{0\})$;
		\item $U(q,p)$ is even and homogeneous of degree $2$ in $p$;
		\item $\Sigma = U^{-1}(1)$ and $\Sigma$ is a regular level surface for $U$.
	\end{itemize}
\end{lemma}
\begin{proof}
	Since $H$ is convex with respect to $p$,
	for every $q_0 \in D$, the set
	\[
	\left\{p: H(q_0,p) = E\right\}
	\]
	is a 
	nonempty, convex, compact hypersurface in $\mathbb{R}^{n} $, symmetric about the origin.
	As a consequence, there exists a unique function
	$U\colon \mathbb{R}^{2n} \to \mathbb{R}$ 
	which is homogeneous of degree $2$ in $p$ and which is identically $1$ on $\Sigma$.
	Since $H$ is of class $C^2$, so it is $U$ on $D \times (\mathbb{R}^n\backslash\{0\})$.
	Moreover, the homogeneity of degree $2$ in $p$ implies both the $C^1$-regularity of $U$ and
	that $U'(q,p) \ne 0$ for all $(q,p) \in \Sigma$.
\end{proof}

Since $\Sigma$ is a regular level surface for the Hamiltonian functions $H$ and $U$,
we have the following result.
\begin{lemma}
	\label{lem:timeParam-HtoU}
	A curve
	$(q,p)\colon [0:T] \to D \times \mathbb{R}^n $ is a solution of Hamilton's equations for $H$
	if and only if
	it is a solution of Hamilton's equations for U,
	up to time reparametrization.
\end{lemma}
\begin{proof}
	See \cite[Lemma 2.1]{Weinstein1978}.
\end{proof}
\begin{remark}
	\label{rmk:fromUtoH}
	Let $x = (q,p)\colon [0,S] \to \mathbb{R}^{2n} $ a solution of the Hamilton's equations with Hamiltonian $U$, hence
	\[
	\frac{\diff x}{\diff s} = J U'(x(s)),
	\]
	By Lemma \ref{lem:timeParam-HtoU},
	there exists a function $\lambda\colon [0,T] \to [0,S]$ such
	that $z = x \circ \lambda\colon [0,T] \to \mathbb{R}^{2n}$
	is a solution of Hamilton's equations with Hamiltonian $H$, hence
	\[
	\frac{\diff z}{\diff t} = JH'(z(t)).
	\]
	As a consequence, we obtain
	\begin{multline*}
	JH'(x(\lambda(t))) = JH'(z(t))
	= \frac{\diff z}{\diff t}(t) 
	= \frac{\diff \lambda}{\diff t}(t)
	\frac{\diff x}{\diff s}(\lambda(t)) \\
	= \frac{\diff \lambda}{\diff t}(t) JU'(x(\lambda(t))).
	\end{multline*}
	Imposing that $\lambda$ is an orientation preserving reparametrization,
	we obtain
	\[
	\frac{\diff \lambda}{\diff t}(t) = 
	\frac{ \sprod{H'(z(t)),U'(z(t))} }{\norm{U'(z(t))}^{2}}=
	\frac{\norm{H'(z(t))}}{\norm{U'(z(t))}}.
	\]
	Hence we have
	\begin{equation}
	\label{eq:fromZtoX}
	\frac{\diff x}{\diff s}(s) =
	\frac{\norm{U'(z(t))}}{\norm{H'(z(t))}}
	\
	\frac{\diff z}{\diff t}(t),
	\quad \text{with } \lambda(t) = s.
	\end{equation}
	The inverse function of $\lambda$ satisfies
	\[
	\frac{\diff(\lambda^{-1})}{\diff s}(s) = 
	\frac{\norm{U'(x(s))}}{\norm{H'(x(s))}},
	\]
	hence we can obtain the time reparametrization solving 
	the following integral
	\begin{equation}
	\label{eq:timeRep}
	t(s) = \int_{0}^{s}
	\frac{\norm{U'(x(\sigma))}}{\norm{H'(x(\sigma))}}\diff\sigma.
	\end{equation}
	
\end{remark}

\bigskip
The following result provides the Finsler metric we will employ in our study.
\begin{lemma}
	\label{lem:def-G}
	Let $G\colon  D\times \mathbb{R}^n\to \mathbb{R}$ be the Legendre transform of $U$
	with respect to $p$, hence
	\begin{equation*}
	\label{eq:defG}
	G(q,v) = \max_{p \in \mathbb{R}^{n} } 
	\left(\sprod{v, p} - U(q,p)\right),
	\end{equation*}
	and define
	$\mathcal{L}\colon D\times\mathbb{R}^n\to D\times\mathbb{R}^n$ as
	\begin{equation}
	\label{eq:defLcal}
	\mathcal{L}(q,p) =
	\left(q,\frac{\partial U}{\partial p}(q,p)\right) =
	(q,v).
	\end{equation}
	Then, 
	the function 
	$F\colon  D\times \mathbb{R}^n\to \mathbb{R}$, defined as 
	$F = \sqrt{G(q,v)}$,
	is a Finsler metric on $D$.
	Moreover, 
	a curve $q\colon [0,S]\to D$ is a Finsler geodesic parametrized by arc length if and only if
	$(q(s),p(s)) = \mathcal{L}^{-1}(q(s),\dot{q}(s)) $
	is a solution of Hamilton's equations with Hamiltonian $U$
	and $U(q(s),p(s)) \equiv 1$.
\end{lemma}
\begin{proof}
	By the convexity of $U$,
	the function $G$ is well defined, convex and homogeneous of degree 2 in $v$.
	Moreover, $G$ is of class $C^2$ on $D \times \mathbb{R}^{n}\backslash\{0\}$,
	while it is of class $C^1$ on $D \times \mathbb{R}^{n}$.
	Thus, the function $F\colon D \times \mathbb{R}^n \to R$ defined as
	$F(q,v) = \sqrt{G(q,v)} $ is a Finsler metric on $D$.
	Since
	\[
	\frac{\partial^2 U}{\partial p^2}(q,p) > 0,
	\qquad \forall q \in D,\ \forall p \in \mathbb{R}^{n} \backslash \left\{0\right\},
	\]
	the map
	$(\partial U/\partial p)(q,\cdot) $ is invertible, 
	thus $ \mathcal{L}$ is a diffeomorphism 
	and it is homogeneous of degree $1$ with respect to $p$.
	The equivalence between the Finsler geodesics parametrized by arc length
	and the solutions of \eqref{eq:hamiltonianSys} with energy $E$ is a 
	direct consequence of the Legendre transform
	(see, for instance, \cite[Chapter I, p. 22]{rund1959}).
\end{proof}

By Lemma \ref{lem:timeParam-HtoU} and Lemma \ref{lem:def-G},
if $q$ is a geodesic in $D$,
then $\mathcal{L}^{-1}(q,\dot{q})$ is a solution of Hamilton's equations with the original Hamiltonian $H$,
up to a time reparametrization.
As a consequence, finding a Finsler geodesic in $D$
is equivalent to finding a solution of \eqref{eq:hamiltonianSys} in $D$
with energy $E$.
The following result provides the reparametrization that links the 
geodesics to the solutions of \eqref{eq:hamiltonianSys},
combining the time reparametrization \eqref{eq:timeRep}
with the Legendre transform defined in \eqref{eq:defLcal}.
For the sake of presentation, we use the following notation
\begin{equation}
\label{eq:def-phi}
\phi(q,v)\coloneqq
\frac{\norm{U'(\mathcal{L}^{-1}(q,v))}}
{\norm{H'(\mathcal{L}^{-1}(q,v))}},
\quad \forall (q,v) \in \mathcal{L}(\Sigma).
\end{equation}
\begin{remark}
	The function $\phi$ given by \eqref{eq:def-phi} is well defined.
	Indeed, since $\overline{\Sigma}$ is compact,
	$\frac{\partial V}{\partial q}(q) \ne 0$ for all $q\in \partial D$
	and $H$ is strictly convex with respect to $p$,
	there exist two constants $h_1,h_2$ such that
	\begin{equation}
	\label{eq:bound-Hprime}
	0 < h_1 \le \norm{H'(q,p)} \le h_2,
	\quad \forall (q,p) \in \overline{\Sigma}.
	\end{equation}
\end{remark}
\begin{lemma}
	\label{lem:fromGtoH}
	Let $\gamma\colon [0,1] \to D$ be a Finsler geodesic
	such that
	\[
	G(\gamma(s),\dot{\gamma}(s)) = c_\gamma,
	\quad \forall s \in [0,1],
	\]
	and let $\lambda\colon [0,T] \to [0,1]$ be the reparametrization such that
	$\gamma \circ \lambda$ is a solution of \eqref{eq:hamiltonianSys}
	with energy $E$.
	Then the inverse of $\lambda$ is given by
	\begin{equation*}
	\label{eq:TimeRep-Finsler}
	t(s) = \sqrt{c_\gamma}
	\int_{0}^s
	\phi\left(\gamma(\sigma),\frac{\dot{\gamma}(\sigma)}{\sqrt{c_\gamma}}\right)
	\diff\sigma.
	\end{equation*}
\end{lemma}
\begin{proof}
	If $G(\gamma(s),\dot{\gamma}(s)) = c_\gamma$ for all $s$, then the reparametrization
	\[
	\lambda_1\colon [0,\sqrt{c_\gamma}]\to [0,1],
	\quad \lambda_1(\tau) =\frac{\tau}{c_\gamma},
	\]
	is such that the curve
	$\hat\gamma = \gamma \circ \lambda_1\colon [0,\sqrt{c_\gamma}]\to D$
	is a geodesic parametrized by arc length.
	By Lemma \ref{lem:def-G},
	the curve $x\colon [0,\sqrt{c_\gamma}] \to D$ defined as
	\[
	x(\tau) = (q(\tau),p(\tau)) = \mathcal{L}^{-1}(\hat\gamma(\tau),\dot{\hat\gamma}(\tau)),
	\]
	is a solution of Hamilton's equations with respect to $U$.
	Let $\lambda_2\colon [0,T] \to [0,\sqrt{c_\gamma}]$ be the inverse of
	\[
	\tau \mapsto
	\int_{0}^{\tau}
	\frac{\norm{U'(x(u))}}{\norm{H'(x(u))}}\diff u
	= \int_{0}^{\tau}
	\phi\left( \hat\gamma(u),\dot{\hat\gamma}(u) \right) \diff u.
	\]
	With the change of variable $\sigma = u/\sqrt{c_\gamma}$,
	$\sigma \in [0,1]$,
	we have
	\begin{multline*}
	\lambda_2^{-1}(\tau) = \sqrt{c_\gamma}\int_{0}^{\tau/\sqrt{c_\gamma}}
	\phi\left(\hat\gamma(\sqrt{c_\gamma} \sigma ),
	\dot{\hat\gamma}(\sqrt{c_\gamma} \sigma )\right)
	\diff \sigma \\
	= \sqrt{c_\gamma} 
	\int_{0}^{\tau/\sqrt{c_\gamma}} \phi \left(\gamma(\sigma),
	\frac{
		\dot{\gamma}(\sigma)}{\sqrt{c_\gamma}}
	\right)
	\diff \sigma.
	\end{multline*}
	By Remark \ref{rmk:fromUtoH},
	in particular by \eqref{eq:timeRep},
	$x \circ \lambda_2$ is a solution of \eqref{eq:hamiltonianSys}.
	As a consequence,
	since $\mathcal{L}^{-1}$ is the identity map with respect the first variable,
	the curve 
	$\gamma \circ \lambda_1 \circ \lambda_2
	\colon [0,T] \to D$
	is the reparametrization of $\gamma$ such that it is a solution of 
	\eqref{eq:hamiltonianSys} with energy $E$.
	Hence, the desired reparametrization $\lambda\colon [0,T] \to [0,1]$
	is given by 
	$
	\lambda = \lambda_1 \circ \lambda_2
	$
	and its inverse $t\colon [0,1] \to [0,T]$ is given by
	\[
	t(s) = \lambda_2^{-1}\left(\lambda_1^{-1}(s)\right)
	= \lambda_2^{-1}(\sqrt{c_{\gamma}} s)
	= 
	\sqrt{c_\gamma} 
	\int_{0}^{s} \phi \left(\gamma(\sigma),
	\frac{
		\dot{\gamma}(\sigma)}{\sqrt{c_\gamma}}
	\right)
	\diff \sigma,
	\]
	and we are done.
	
\end{proof}

\begin{remark}
	\label{rmk:MaupertuisRiemannian}
	When $H$ is a Hamiltonian of natural type,
	the previous construction leads to the well-known Maupertuis principle
	(cf. \cite{Giambo2004}).
	Indeed, set
	\begin{equation*}
	\label{eq:def-natHamiltonian2}
	H(q,p) = \frac{1}{2}\sum_{i,j = 1}^n a^{ij}(q) p_i p_j + V(q),
	\end{equation*}
	where $ (a^{ij}(q))$ is a positive definite quadratic form on $\mathbb{R}^n$.
	Then, using the construction above, 
	\begin{equation*}
	\label{eq:U-Riemannian}
	U(q,p) = \frac{1}{2(E - V(q))}\sum_{i,j = 1}^n a^{ij}(q) p_i p_j,
	\end{equation*}
	and its Legendre transform is
	\begin{equation}
	\label{eq:G-Riemannian}
	G(q,v) = \frac{1}{2} (E - V(q))\sum_{i,j = 1}^{n} a_{ij}(q)v^i v^j,
	\end{equation}
	where $(a_{ij}(q))$ is the inverse of 
	$(a^{ij}(q))$.
	We observe that $G(q,v)$ degenerates on the boundary $\partial D$,
	where, by continuity, it can be extended to $0$. 
	Since
	\[
	U'(q,p) = \frac{1}{E - V(q)} H'(q,p),
	\quad \forall (q,p )\in \Sigma,
	\]
	then 
	\[
	\frac{\norm{U'(q,p)}}{\norm{H'(q,p)}}
	= \frac{1}{E - V(q)},
	\quad \forall (q,p )\in \Sigma.
	\]
	Using Lemma \ref{lem:fromGtoH},
	if $\gamma\colon [0,1] \to D$ is a geodesic of constant speed
	with respect to the Riemannian metric $\sqrt{G}$,
	then we can obtain the reparametrization $\lambda\colon [0,T] \to [0,1]$ such that 
	$q = \gamma\circ \lambda\colon [0,T] \to D$ is a solution of \eqref{eq:hamiltonianSys} for $H$.
	Using \eqref{eq:timeRep},
	the inverse of $\lambda$ is given by
	\begin{equation*}
	\label{eq:partTimeRep-Riemannian}
	t(s) = \sqrt{c_\gamma}\int_{0}^{s} \frac{1}{E - V(\gamma(\sigma))}\diff\sigma,
	\end{equation*}
	where $G(\gamma_y(s),\dot{\gamma}(s))\equiv c_\gamma$.
\end{remark}

\section{Jacobi-Finsler metric near the boundary}
Since $U$ and $G$ are not defined on $\partial D$,
the above construction does not allow to see
the brake orbits in $\overline{D}$ as Finsler geodesics.
In the following, we estimate the behaviour of $U$ near the boundary $\partial D$
and we will show that  
$G$ degenerates on $\partial D$ to the zero function,
as it can be seen in \eqref{eq:G-Riemannian} for the case of natural Hamiltonian systems.
Differently from \cite{Weinstein1978},
we are interested in the multiplicity of the brake orbits,
not only in their existence.
Hence,
in addition to the construction given in \cite{Weinstein1978}, 
we give an upper and a lower bound for the Finsler metric
which depend only on $H$ and the potential well $D$,
and these bounds will be exploited to obtain the one-one correspondence
between the brake orbits and the orthogonal geodesic chords.

As a preliminary step, we give the following result,
which is available up to a modification of $H(q,p) = K(q,p) + V(q)$ far away from $\Sigma$.
\begin{lemma}
	\label{lem:K-bounds}
	There exist two constants $\nu_1,\nu_2 > 0 $ such that the followings hold
	for every $q \in \overline{D}$ and $p \in \mathbb{R}^{n}$:
	\begin{equation}
	\label{eq:bound-Hessian-K-nu}
	\nu_1\norm{\xi}^2
	\le
	\frac{\partial^2 K}{\partial p^2}(q,p)[\xi,\xi] 
	\le \nu_2 \norm{\xi}^2,
	\quad \forall \xi \in \mathbb{R}^n;
	\end{equation}
	\begin{equation}
	\label{eq:bound-deriv-K-nu}
	\nu_1 \norm{p}
	\le  \norm*{\frac{\partial K}{\partial p}(q,p)} 
	\le \nu_2 \norm{p};
	\end{equation}
	\begin{equation}
	\label{eq:bound-K-nu}
	\frac{1}{2}\nu_1 \norm{p}^2 
	\le K(q,p) 
	\le \frac{1}{2}\nu_2 \norm{p}^2.
	\end{equation}
\end{lemma}
\begin{proof}
	Since we are interested on the solutions of Hamilton's equations for $H$ in $\Sigma$, which is a bounded set,
	we can modify $H$ far away from $\Sigma$.
	Hence, we may assume that $H$ is fiber-wise quadratic for $\norm{p}$ sufficiently large.
	By \eqref{eq:defPosDef} and the compactness of $\overline{D}$,
	there exist 
	$\nu_1,\nu_2 > 0$ such that 
	\eqref{eq:bound-Hessian-K-nu} holds.
	Since $(\partial K/\partial p)(q,0) = 0$  and $K(q,0) = 0$ for all $q$,
	from \eqref{eq:bound-Hessian-K-nu} we infer
	\eqref{eq:bound-deriv-K-nu} and 
	\eqref{eq:bound-K-nu} by integration.
\end{proof}

\begin{lemma}
	\label{lem:degeneration-on-pD}
	Let $\nu_1,\nu_2$ the constants defined by Lemma \ref{lem:K-bounds}.
	Then, the followings hold:
	\begin{equation}
	\label{eq:bound-U}
	\frac{\nu_1}{2(E-V(q))}\norm{p}^2
	\le U(q,p)
	\le \frac{\nu_2}{2(E-V(q))}\norm{p}^2,
	\quad \forall (q,p)\in D \times \mathbb{R}^n,
	\end{equation}
	and
	\begin{equation}
	\label{eq:bound-G}
	\frac{(E-V(q))}{2\nu_2}\norm{v}^2
	\le  G(q,v)
	\le \frac{(E-V(q))}{2\nu_1}\norm{v}^2,
	\quad \forall (q,v)\in D \times \mathbb{R}^n.
	\end{equation}
	Moreover, there exists a constant $\nu_3 > 0$ such that
	\begin{equation}
	\label{eq:bound-Uprime}
	\norm{U'(q,p)} \ge \frac{\nu_3}{E - V(q)},
	\quad \forall (q,p) \in \Sigma.
	\end{equation}
	
\end{lemma}
\begin{proof}
	Set $\mathbb{S}^{n-1}=\left\{\theta\in\mathbb{R}^{n}:\norm{\theta}=1\right\}$.
	We define 
	$\widetilde{H}\colon D\times\mathbb{S}^{n-1}\times\mathbb{R}^{+}\to\mathbb{R}$
	as
	\[
	\widetilde{H}(q,\theta,\omega) = H(q,\omega\theta) - E.
	\]
	Since $\widetilde{H}(q,\theta,0) < 0$
	for all $(q,\theta)\in D \times \mathbb{S}^{n-1}$,
	exploiting also the convexity of $H$ we obtain that 
	for all $(q,\theta) \in D \times \mathbb{S}^{n-1}$
	there exists an unique $\omega > 0$ such that 
	$\widetilde{H}(q,\theta,\omega) = 0$.
	As a consequence, 
	the function $\omega\colon D \times \mathbb{S} ^{n-1} \to \mathbb{R}^+$
	such that
	\[
	\widetilde{H}(q,\theta,\omega(q,\theta)) = 0,
	\quad\forall (q,\theta) \in D\times \mathbb{S}^{n-1}, 
	\]
	is well defined.
	Moreover, by \eqref{eq:bound-deriv-K-nu} we have
	\[
	\frac{\partial \widetilde{H}}{\partial \omega}(q,\theta,\omega)
	= \sprod*{\frac{\partial H}{\partial p}(q,\omega \theta),\ \theta}
	= \frac{1}{\omega}
	\sprod*{\frac{\partial K}{\partial p}(q,\omega \theta),\omega \theta}
	\ge \omega \nu_1	 > 0,
	\]
	so we can apply the implicit function theorem 
	to obtain that
	the function $\omega$ is of class $C^2$ and it satisfies
	\begin{equation}
	\label{eq:deriv-f-q}
	\begin{multlined}
	\frac{\partial \omega}{\partial q}(q,\theta) = 
	-
	\left( \sprod*{
		\frac{\partial H}{\partial p}(q,\omega(q,\theta)\theta), \theta}
	\right)^{-1}
	\frac{\partial H}{\partial q}(q,\omega(q,\theta)) \\
	=
	-
	\left( \sprod*{
		\frac{\partial K}{\partial p}(q,\omega(q,\theta)\theta), \theta}
	\right)^{-1}
	\frac{\partial H}{\partial q}(q,\omega(q,\theta)).
	\end{multlined}
	\end{equation}
	By definition of $\omega(q,\theta)$,
	$K(q,\omega(q,\theta)\theta) = E - V(q)$
	for all $(q,\theta)\in D \times \mathbb{S}^{n-1}$.
	By \eqref{eq:bound-K-nu}, we have
	\begin{equation*}
	\label{eq:deg-onPd-proof1}
	\frac{1}{2}\nu_1 \omega^2(q,\theta) 
	\le K(q,\omega(q,\theta)\ \theta) = E- V(q)
	\le \frac{1}{2}\nu_2 \omega^2(q,\theta),
	\end{equation*}
	hence
	\begin{equation}
	\label{eq:bound-f}
	0 < \frac{2(E - V(q))}{\nu_2} 
	\le \omega^2(q,\theta)
	\le \frac{2(E - V(q))}{\nu_1},
	\qquad \forall q \in D, \ \forall \theta \in \mathbb{S}^{n-1}.
	\end{equation}
	By definition of $\omega(q,\theta)$, we have also
	\[
	U(q,\omega(q,\theta)\theta) = 1,
	\qquad \forall q \in D,\ \forall \theta \in \mathbb{S}^{n-1}.
	\]
	Since $U$ is homogeneous of degree $2 $ in $p$,
	for all $q \in D$ and $p \ne 0$ we obtain
	\begin{equation}
	\label{eq:U-by-r}
	U(q,p) = 
	\frac{\norm{p}^{2}}{\omega^{2}\left(q,p/\norm{p}\right)}
	U\left(q,\omega\left(q,\frac{p}{\norm{p}}\right)\ \frac{p}{\norm{p}}\right)
	=
	\frac{\norm{p}^{2}}{\omega^{2}\left(q,p/\norm{p}\right)}.
	\end{equation}
	Using \eqref{eq:bound-f} and \eqref{eq:U-by-r},
	we obtain \eqref{eq:bound-U}.
	Since the Legendre transform inverts the order relation,
	\eqref{eq:bound-U} implies \eqref{eq:bound-G}.
	
	It remains to prove \eqref{eq:bound-Uprime}.
	Let us fix $\delta > 0$ and set
	\[
	D_\delta =
	\left\{q \in D: V(q) \le E - \delta\right\},
	\]
	so every point in $D_\delta$ is far away from the boundary $\partial D$.
	By the bounds on the function $U$ given by \eqref{eq:bound-U}
	and recalling that $U$ is homogeneous of degree $2$ in $p$, 
	there exists a constant $c_\delta$ such that
	\begin{equation}
	\label{eq:bound-Uprime-proof3}
	\norm{U'(q,p)} \ge \norm*{\frac{\partial U}{\partial p}(q,p)}
	\ge c_\delta > 0,
	\quad \forall (q,p) \in \Sigma,\ q \in D_\delta.
	\end{equation}
	By the arbitrariness of $\delta$,
	we can obtain \eqref{eq:bound-Uprime} by proving it 
	for all $(q,p) \in \Sigma$ with $q$ sufficiently near the boundary.
	More precisely, we prove
	the existence of a constant $c_1$ 
	such that 
	\begin{equation}
	\label{eq:bound-Uprime-proof2}
	\norm{U'(q,p)} \ge
	\norm*{\frac{\partial U}{\partial q}(q,p)} \ge
	\frac{c_1}{E - V(q)},
	\end{equation}
	for all $(q,p) \in \Sigma$ with $q$ sufficiently near the boundary.
	For every $(q,p) \in \Sigma$, $U(q,p) = 1 $,
	so by \eqref{eq:U-by-r} we obtain
	\begin{multline*}
	\label{eq:dU-dq}
	\frac{\partial U}{\partial q}(q,p)
	= 
	-
	\frac{2 \norm{p}^{2}}{\omega^3(q,p/\norm{p})}
	\frac{\partial \omega}{\partial q}\left(q,\frac{p}{\norm{p}}\right) \\
	=
	- \frac{2}{\omega(q,p/\norm{p})}
	\frac{\partial \omega}{\partial q}\left(q,\frac{p}{\norm{p}}\right),
	\quad \forall (q,p)\in \Sigma.
	\end{multline*}
	As a consequence, using also \eqref{eq:deriv-f-q}
	and denoting $p/\norm{p}$ by $\theta$, we have
	\begin{equation}
	\label{eq:dU-dq-Sigma}
	\frac{\partial U}{\partial q}(q,p) = 
	\frac{2}{\omega(q,\theta)}
	\left( \sprod*{
		\frac{\partial K}{\partial p}(q,\omega(q,\theta)\theta), \theta}
	\right)^{-1}
	\frac{\partial H}{\partial q}(q,\omega(q,\theta)),
	\end{equation}
	for all $(q,p) \in \Sigma$.
	By \eqref{eq:bound-deriv-K-nu}, we have
	\[
	\norm*{ \sprod*{
			\frac{\partial K}{\partial p}(q,\omega(q,\theta)\theta), \theta}
	}^{-1}
	\ge \frac{1}{\nu_2 \omega(q,\theta)}.
	\]
	Hence, by \eqref{eq:dU-dq-Sigma} and using again \eqref{eq:bound-f},
	we obtain
	\begin{multline}
	\label{eq:boundU-prime-proof1}
	\norm*{\frac{\partial U}{\partial q}(q,p)}
	\ge 
	\frac{2}{\nu_2 \omega^2(q,\theta)}
	\norm*{\frac{\partial H}{\partial q}(q,\omega(q,\theta)\theta)}\\
	\ge
	\frac{\nu_1}{\nu_2(E- V(q))}
	\norm*{\frac{\partial H}{\partial q}(q,\omega(q,\theta)\theta)},
	\quad \forall (q,p)\in \Sigma.
	\end{multline}
	The existence of a strictly positive constant $c_1$
	such that \eqref{eq:bound-Uprime-proof2} holds
	for all $(q,p) \in \Sigma$ with $q$ sufficiently near the boundary  
	$\partial D$
	can be obtained
	by \eqref{eq:boundU-prime-proof1},
	recalling that $\overline{D}$ is compact and
	$(\partial V/\partial q)(q) \ne 0$ in $\partial D$.
	Finally,	we obtain \eqref{eq:bound-Uprime}
	by \eqref{eq:bound-Uprime-proof3}
	and \eqref{eq:bound-Uprime-proof2},
	recalling the arbitrariness of $\delta$.
\end{proof}

\begin{remark}
	By \eqref{eq:bound-G}, we can extend $G$ on the boundary $\partial D$ by continuity.
	Denoting this extension again with $G$, we have 
	\[
	G(q,v) = 0,
	\quad \forall q \in \partial D, \forall v \in \mathbb{R}^n.
	\]
\end{remark}

\subsection{Behaviour of the solutions near the boundary}

In this section, we present some preliminary results
about the behaviour of the solutions of Hamilton's equations
near the boundary of the potential well $D$.
These results are required to analyse the time reparametrization of the Finsler geodesics
that correspond to the brake orbits and
to study the concavity of the set $\overline{\Omega}$
described in Theorem \ref{teo:main}.

\begin{lemma}
	\label{lem:d2Vepsilon}
	There exists $\bar\epsilon > 0$ such that,
	if $(q(t),p(t))$ is a solution of \eqref{eq:hamiltonianSys}
	with Hamiltonian $H$ and 
	energy $E$ such that $V(q(t))\ge E - \bar\epsilon$
	for $t \in [a,b]$, then
	\begin{equation*}
	\label{eq:d2Vepsilon}
	\frac{\diff^2}{\diff t^2} V(q(t)) \le - \bar\epsilon,
	\quad \forall t \in [a,b].
	\end{equation*}
\end{lemma}
\begin{proof}
	See \cite[Lemma 5.2]{Weinstein1978}.
\end{proof}
The following result provides an upper bound for the length of a time interval in which a solution of Hamilton's equations with energy $E$ can be uniformly near the boundary.
\begin{lemma}
	\label{lem:boundedTime}
	Let $\bar\epsilon$ given by Lemma \ref{lem:d2Vepsilon}.
	If $(q(t),p(t))$ is a solution of Hamilton's equations
	with total energy $E$
	and $V(q(t))\ge E - \bar\epsilon /2$
	for $t \in [a,b]$, then $b - a \le 2$.
\end{lemma}
\begin{proof}
	See \cite[Corollary 5.3]{Weinstein1978}.
\end{proof}

For every $Q \in \partial D$,
we denote by $z(t,Q) = (q(t,Q),p(t,Q))$
the solution of Hamilton's equations for $H$
with total energy $E$ and such that $q(0) = Q$.
Since $z(t,Q)$ is the solution of the Cauchy problem
\[
\begin{cases}
\dot{z}(t,Q) = J H'(z(t,Q)),\\
z(0,Q) = (Q,0),
\end{cases}
\]
it is well defined and of class $C^1$.
\begin{remark}
	\label{rmk:dzC1}
	Since $JH'$ is a function of class $C^1$,
	also $\dot{z}(t,Q)$ is of class $C^1$ with respect to the variables $t$ and $Q$.
\end{remark}

\begin{lemma}
	\label{lem:dotq-boundarySolutions}
	For every $Q_0 \in \partial D$,
	there exists
	a function
	$\rho\colon[0,+\infty[\times \partial D \to\mathbb{R}^n $
	of class $C^1$ such that
	$\diff\rho(0,Q_0) = 0$ and
	\begin{equation}
	\label{eq:dotq-nearBoundary}
	\dot{q}(t,Q) = -t \frac{\partial^2 H}{\partial p^2}(Q_0,0)\nabla V(Q_0) 
	+ \rho(t,Q),
	\quad \forall t \in [0,+\infty[,\ \forall Q \in \partial D.
	\end{equation}
\end{lemma}
\begin{proof}
	We define 
	$\rho_0\colon [0,+\infty[\times \partial D \to \mathbb{R}^n$
	as 
	\[
	\rho_0(t,Q) = \ddot{q}(t,Q) - \ddot{q}(0,Q_0).
	\]
	Recalling that $z(t,Q)$ is of class $C^1$ both respect to $t$ and $Q$,
	and taking the derivative with respect to $t$ of
	$\dot{z}(t,Q) = JH'(z(t,Q))$, we obtain that 
	$\ddot{z}(t,Q)$ is a continuous function,
	so $\rho_0(t,Q)$ is a continuous function.
	Since $p(0,Q) = 0$ for all $Q \in \partial D$, then $\dot{q}(0,Q) = 0$ and
	\[
	\begin{multlined}
	\ddot{q}(0,Q)
	= \frac{\diff}{\diff t}\dot{q}(t,Q)\bigg|_{t = 0}
	= \bigg(
	\frac{\partial^2 H}{\partial q \partial p}(q(t,Q),p(t,Q))\dot{q}(t,Q)\\
	+ \frac{\partial^2 H}{\partial p^2}(q(t,Q),p(t,Q))\dot{p}(t,Q)
	\bigg)\bigg|_{t = 0}\\
	=  -\frac{\partial^2 H}{\partial p^2}(Q,0)\nabla V(Q),
	\quad \forall Q \in \partial D.
	\end{multlined}
	\]
	By definition of $\rho_0$, we have
	\[
	\ddot{q}(t,Q) = \ddot{q}(0,Q_0) + \rho_0(t,Q) = 
	-\frac{\partial^2 H}{\partial p^2}(Q_0,0)\nabla V(Q_0) + \rho_0(t,Q).
	\]
	Integrating the previous equation and recalling that $\dot{q}(0,Q)= 0$, 
	we obtain \eqref{eq:dotq-nearBoundary}, with $\rho(t,Q) = \int_0^t \rho_0(\tau,Q)d\tau$.
	Since $\rho_0(0,Q_0) = 0$, we have
	$\diff\rho(0,Q_0)= 0$.
\end{proof}

\begin{lemma}
	\label{lem:def-bardelta}
	There exists a constant $\bar\delta > 0$ such that 
	the following property holds:
	\begin{equation}
	\label{eq:def-bardelta}
	\begin{split}
	&\text{
		$\forall y \in D$ with 
		$\text{dist}(y,\partial D) \le \bar\delta$
		there exists a unique solution $(q_y,p_y)$
		of \eqref{eq:hamiltonianSys}
	}\\
	&\text{
		with energy $E$
		and a unique $t_y > 0$
		such that $q_y(0) \in \partial D$, $q_y(t_y) = y$ 
		and
	}\\
	&\text{
		$\text{dist}(q_y(t),\partial D) \le \bar\delta$,
		for all $t \in [0,t_y]$.}
	\end{split}
	\end{equation}
\end{lemma}

\begin{proof}
	By Lemma \ref{lem:dotq-boundarySolutions},
	for every $Q_0 \in \partial D$ there exists a 
	function $\rho\colon \partial D \times \mathbb{R} \to \mathbb{R}^n$
	such that
	\begin{equation*}
	\label{eq:qtQ-nearBoundary}
	q(t,Q) 
	= Q - \frac{t^2}{2}\
	\frac{\partial^2 H}{\partial p^2}(Q_0,0)\frac{\partial V}{\partial q}(Q_0) 
	+ \int_0^t\rho(\tau,Q)\diff\tau,
	\quad \forall Q \in \partial D,
	\end{equation*}
	where the vector 
	$
	\frac{\partial^2 H}{\partial p^2}(Q_0,0)\frac{\partial V}{\partial q}(Q_0) 
	$
	is not tangent to $\partial D$ for every $Q_0 \in \partial D$.
	Indeed, $\frac{\partial V}{\partial q}(Q_0)$ is 
	orthogonal to $\partial D$ by definition of $D$ 
	and by \eqref{eq:bound-Hessian-K-nu} we have
	\[
	\sprod*{
		\frac{\partial^2 H}{\partial p^2}(Q_0,0)\frac{\partial V}{\partial q}(Q_0) 
		,
		\frac{\partial V}{\partial q}(Q_0)
	}
	\ge 
	\nu_1 \norm*{\frac{\partial V}{\partial q}(Q_0)}^{2} > 0,
	\quad \forall Q_0 \in \partial D.
	\]
	As a consequence,
	if $\left\{y_1,\dots,y_{n-1}\right\}$ is a coordinate system of 
	$\partial D $ in a neighbourhood of $Q_0$,
	then
	$\left\{y_1,\dots,y_{n-1},t\right\}$
	is a local coordinate system on the manifold 
	with boundary $\partial D$ and
	$(t,Q)\mapsto q(t,Q)$ defines a local chart.
	By the compactness of $\partial D$,
	we obtain the existence of a $\bar\delta > 0$ which satisfies
	\eqref{eq:def-bardelta}.
\end{proof}

\textbf{Notation:}
If $y \in D$ is such that $\text{dist}(y,\partial D)\le \bar\delta$,
we denote by $(t_y,Q_y)$ the unique 
element in $\mathbb{R}^+ \times \partial D$ such that 
$q(t_y,Q_y) = y$ and 
$\text{dist}(q(t,Q_y),\partial D) \le \bar\delta$,
for all $t \in [0,t_y]$.
\begin{remark}
	\label{rmk:tyQyC1}
	Both
	$t_y$ and $Q_y$ are functions of class $C^1$ with respect to $y$,
	since they are implicitly defined by the coordinate system given by 
	the proof of Lemma \ref{lem:def-bardelta}.
\end{remark}

\section{The Jacobi-Finsler energy function}
\label{sec:JF-distance}

We define the functional $\mathcal{J}\colon W^{1,2}([0,1],\overline{D}) \to \mathbb{R}$ as
\begin{equation*}
\label{eq:def-Jcal}
\mathcal{J}(\gamma) = \int_{0}^{1}G(\gamma(s),\dot{\gamma}(s))\diff s.
\end{equation*}
If $\gamma([0,1])\subset D$,
then $\mathcal{J}$ is differentiable at $\gamma$
and its differential
\[
\diff\mathcal{J}(\gamma)\colon W^{1,2}([0,1],\mathbb{R}^n) \to \mathbb{R}
\]
is given by
\begin{equation*}
\label{eq:diff-J}
\diff\mathcal{J}(\gamma)[\xi] = 
\int_{0}^{1} \left(
\frac{\partial G}{\partial q}(\gamma(s),\dot{\gamma}(s))[\xi(s)] 
+ \frac{\partial G}{\partial v}(\gamma(s),\dot{\gamma}(s))[\dot{\xi}(s)]
\right)\diff s.
\end{equation*}
For every $y \in D$, we define $X_y$ as
\[
X_y \coloneqq \left\{
\gamma \in W^{1,2}([0,1],\overline{D}):
\gamma(0) = y,\
\gamma([0,1[) \subset D\
\text{and}\ 
\gamma(1) \in\partial D
\right\}.
\]
\begin{definition}
	We define the function $\psi\colon \overline{D} \to \mathbb{R}$ as
	\begin{equation}
	\label{eq:def-psi}
	\psi(y) \coloneqq \inf_{\gamma \in X_y}\mathcal{J}(\gamma). 
	\end{equation}
\end{definition}
The function $\psi$ will be the main focus of our analysis.
Indeed, from now on we will state and prove some results that 
will lead to define the set
$\Omega$ described in Theorem \ref{teo:main} as $\psi^{-1}(]\delta,\infty[)$,
for some $\delta$ sufficiently small.

\begin{lemma}
	\label{lem:def-d1}
	There exists a constant $\bar{d} > 0$ such that 
	\begin{equation*}
	\label{eq:bound-psi}
	\psi(y) \le \bar{d},
	\quad \forall y \in \overline{D}.
	\end{equation*}
\end{lemma}
\begin{proof}
	The thesis directly follows from the upper bound given in \eqref{eq:bound-G}
	and the compactness of $\overline{D}$.
\end{proof}

\subsection{Existence and uniqueness of the minimum}

In this section,
we prove that $\psi(y)$ is always attained on at least one curve and,
if $y$ is sufficiently near the boundary $\partial D$, this curve is unique.
To these aims, we
construct the minimum curve as a 
limit of a sequence of Finsler geodesics
$(\gamma_k) \subset W^{1,2}([0,1],D)$.
To prove that $(\gamma_k)$ is weakly uniformly convergent in $W^{1,2}([0,1],D)$,
we will exploit the fact that $(T_k)$ is uniformly bounded,
where $T_k$ are given by
\begin{equation*}
\label{eq:def-Tk}
T_k = t_k(1) = 
\sqrt{\mathcal{J}(\gamma_k)}
\int_{0}^1
\phi\left(\gamma(\sigma),\frac{\dot{\gamma}(\sigma)}
{\sqrt{\mathcal{J}(\gamma_k)}}\right)
\diff\sigma.
\end{equation*}
Recalling the reparametrization given by Lemma \ref{lem:fromGtoH},
$T_k$ is the final time of the reparametrization of $\gamma_k$ which 
is a solution of \eqref{eq:hamiltonianSys} with energy $E$.
More formally, we require the following lemma.

\begin{lemma}
	\label{lem:finiteTk}
	Let $(\gamma_k)\subset W^{1,2}([0,1],D)$
	be a sequence of Finsler geodesics.
	If there exist two constants $c_1,c_2$ such that
	\begin{equation}
	\label{eq:finiteTk-hp}
	0 < c_1 \le \mathcal{J}(\gamma_k) \le c_2,
	\quad \forall k \in \mathbb{N},
	\end{equation}
	then there exist two constants $c_3,c_4$ such that
	\begin{equation}
	\label{eq:bounded-param}
	0 < c_3 \le
	T_k
	\le c_4,
	\quad \forall k \in \mathbb{N}.
	\end{equation}

\end{lemma}

\begin{proof}
	Since $(E - V(q))\to 0$ when $q \to \partial D$ 
	and $\overline{D}$ is compact,
	there exists a strictly positive constant $c_5$ such that
	\[
	\frac{1}{E - V(q)} \ge c_5, \quad \forall q \in D.
	\]
	By the definition of $\phi$ given by \eqref{eq:def-phi},
	using  \eqref{eq:bound-Hprime} and \eqref{eq:bound-Uprime}
	we have
	\begin{equation*}
	\int_{0}^1 \phi
	\left(
	\gamma_k(s),\frac{\dot{\gamma}_k(s)}{\sqrt{\mathcal{J}(\gamma_k)}}
	\right)
	\diff s
	\ge
	\int_0^1 \frac{\nu_3}{h_2(E - V(\gamma_k(s)))}\diff s
	\ge
	\frac{c_5 \nu_3}{h_2} > 0.
	\end{equation*}
	Using also \eqref{eq:finiteTk-hp}, we obtain
	\[
	T_k = 
	\sqrt{\mathcal{J}(\gamma_k)}
	\int_{0}^1
	\phi\left(\gamma(\sigma),\frac{\dot{\gamma}(\sigma)}
	{\sqrt{\mathcal{J}(\gamma_k)}}\right)
	\diff\sigma \ge
	\sqrt{c_1}\	\frac{c_5 \nu_3}{h_2}
	\eqqcolon c_3 > 0.
	\]
	To prove the existence of a constant $c_4$ such that 
	\eqref{eq:bounded-param} holds,
	we work directly on the reparametrizations of $\gamma_k$.
	Following the construction given in \cite[Lemma 5.1]{Weinstein1978},
	let $\bar\epsilon > 0$ be given by Lemma \ref{lem:d2Vepsilon}
	and let us divide $\overline{D}$ into:
	\begin{itemize}
		\item 
		the rim 
		$\left\{q \in \overline{D}: V(q)\ge E - \bar{\epsilon}/2\right\}$;
		\item
		the band 
		$\left\{q \in \overline{D}:  E - \bar{\epsilon}\le V(q)\le E - \bar\epsilon /2\right\}$;
		\item
		the core 
		$\left\{q \in \overline{D}: V(q)\le E - \bar\epsilon\right\}$.
	\end{itemize}
	Let us set
	\[
	\overline{\Sigma}_{\bar\epsilon/2} = 
	\left\{(q,p) \in \Sigma: V(q)\le E - \bar\epsilon/2\right\}.
	\]
	Since $\overline{\Sigma}_{\bar\epsilon/2}$ is compact,
	there exists a constant $\bar\phi > 0$ such that
	$\phi(q,p)\le \bar\phi$ for all $(q,p) \in \overline{\Sigma}_{\bar\epsilon/2}$.
	For every $k$, we set
	$I_k = \left\{s \in [0,1]: V(\gamma_k(s))<  E - \bar\epsilon/2\right\}$ and
	$C_k = [0,1]\setminus I_k$.
	Hence,
	\begin{multline}
	\label{eq:finiteParam-proof1}
	\int_{0}^1\phi
	\left( \gamma_k(s),\frac{\dot{\gamma}_k(s)}{\sqrt{\mathcal{J}(\gamma_k)}} \right)
	\diff s   
	= \int_{I_k}\phi
	\left( \gamma_k(s),\frac{\dot{\gamma}_k(s)}{\sqrt{\mathcal{J}(\gamma_k)}} \right)
	\diff s \\
	+
	\int_{C_k}\phi
	\left( \gamma_k(s),\frac{\dot{\gamma}_k(s)}{\sqrt{\mathcal{J}(\gamma_k)}} \right)
	\diff s 
	\le
	\bar\phi + 
	\int_{C_k}\phi
	\left( \gamma_k(s),\frac{\dot{\gamma}_k(s)}{\sqrt{\mathcal{J}(\gamma_k)}} \right)
	\diff s.
	\end{multline}
	For every $k$, the set $C_k$ is the union of closed and disjoint intervals
	in which the orbit $\gamma_k$ is in the rim.
	The orbit can enter the rim many times but,
	as a consequence of Lemma \ref{lem:fromGtoH},
	each pair of passages into the rim must be separated by a dip into the core,
	and this requires the solution to cross the band twice.
	This bounds the number of closed disjoint intervals 
	that constitutes $C_k$, independently of $k$. 
	Indeed, let us set
	\[
	\bar{d} = \min\left\{\mathcal{J}(\gamma): \gamma \in W^{1,2}([0,1],D),
	V(\gamma(0)) = E - \bar\epsilon,
	V(\gamma(1)) = E - \bar\epsilon/2
	\right\}.
	\]
	Since \eqref{eq:finiteTk-hp} holds,
	we have that $\gamma_k$ can cross the band at most
	$N$ times, where $N$ is a positive integer strictly greater than
	$
	c_2/(2\bar{d}),
	$
	independent on $k$. 
	As a consequence, by Lemma \ref{lem:boundedTime} we have 
	\[
	\sqrt{\mathcal{J}(\gamma_k)}
	\int_{C_k}\phi
	\left( \gamma_k(s),\frac{\dot{\gamma}_k(s)}{\sqrt{\mathcal{J}(\gamma_k)}} \right)
	\diff s
	\le
	2N,
	\]
	and, by \eqref{eq:finiteTk-hp} and \eqref{eq:finiteParam-proof1}, we have
	\[
	T_k \le 
	\sqrt{c_2}\ \bar\phi + 2N \eqqcolon c_4,
	\]
	so \eqref{eq:bounded-param} holds.
\end{proof}

\begin{proposition}
	\label{prop:psi-attains-min}
	For every $y \in D$,
	$\psi(y)$ is attained on at least one curve
	$\gamma_y \in X_y$.
	Moreover, $\gamma_y$ satisfies
	\begin{equation}
	\label{eq:cond-min-gammay}
	\int_{0}^{1} \left(
	\frac{\partial G}{\partial q}(\gamma_y(s),\dot{\gamma}_y(s))[\xi(s)] 
	+ \frac{\partial G}{\partial v}(\gamma_y(s),\dot{\gamma}_y(s))[\dot{\xi}(s)]
	\right)\diff s = 0,
	\end{equation}
	for all $\xi \in W^{1,2}_0([0,1],\mathbb{R}^n)$,
	and there exist a $T > 0$ and a diffeomorphism
	$\sigma\colon [0,T] \to [0,1]$ such that,
	setting $\hat\gamma_y = \gamma_y \circ \sigma\colon [0,T] \to \overline{D}$,
	the pair
	$
	(q,p)\colon [0,T] \to \overline{D}\times \mathbb{R}^n
	$
	given by
	\[
	(q(t),p(t)) = \mathcal{L}^{-1}\left(\hat{\gamma}_y(t),\dot{\hat\gamma}_y(t)\right)
	\]
	is a solution of \eqref{eq:hamiltonianSys} with energy $E$,
	$q(0) = y$ and $q(T) \in \partial D$.
\end{proposition}
\begin{proof}
	For every $k \in \mathbb{N} $, set
	$D_k = V(]-\infty,E - 1/k[)\cap D$ and
	\[
	X_y^k = \left\{
	\gamma \in W^{1,2}([0,1],\overline{D_k}):
	\gamma(0) = y,
	\gamma(1) \in \partial D_k
	\right\}.
	\]
	For $k$ sufficiently large,
	$y \in D_k$
	and $D_k$ is homeomorphic to $D$.
	By standard arguments,
	it can be proved that,
	for $k$ sufficiently large,
	the problem of minimization of the functional $\mathcal{J}$ in the space $X_y^k$
	has a solution $\gamma_k$ which
	is a Finsler geodesic and such that	$\gamma_k([0,1[)\subset D_k$.
	Setting $\ell_k = \mathcal{J}(\gamma_k)$,
	by definition of $\psi$ in \eqref{eq:def-psi} 
	we have
	\[
	\liminf_{k \to \infty}\ell_k \ge \psi(y).
	\]
	We claim that
	\begin{equation}
	\label{eq:lkToPsiY}
	\liminf_{k \to \infty}\ell_k = \psi(y).
	\end{equation}
	By absurd, 
	if it was $\liminf_{k\to \infty}\ell_k > \psi(y)$,
	then we could find a curve $x \in X_y$ 
	such that $\mathcal{J}(x)< \liminf_{k \to \infty}\ell_k$
	and a suitable reparametrization of $x$ would yield a curve
	$x_k \in X_y^k$ such that $\mathcal{J}(x_k)<\ell_k$,
	which contradicts the minimality of $\ell_k$.
	Hence \eqref{eq:lkToPsiY} holds.
	Since $\gamma_k$ minimizes $\mathcal{J}$ on $X^k_y$,
	it is a geodesic with constant speed, hence 
	\[
	G(\gamma_k(s),\dot{\gamma}_k(s)) = \ell_k,
	\quad \forall s \in [0,1].
	\]
	Using \eqref{eq:bound-G} and Lemma \ref{lem:def-d1},
	there are two constants $c_1,c_2$ such that
	\begin{equation}
	\label{eq:bounded-lk}
	0 < c_1 \le \ell_k \le c_2,
	\end{equation}
	for all $k$ sufficiently large.
	As a consequence, 
	we can apply Lemma \ref{lem:finiteTk},
	so there exist two constants $c_3,c_4$ such that 
	\eqref{eq:bounded-param} holds 
	for every $k$ sufficiently large.
	Using also  \eqref{eq:bound-Hprime}, \eqref{eq:bound-Uprime} and 
	\eqref{eq:bounded-lk},
	we have
	\begin{equation*}
	\label{eq:phi-EVq}
	c_4 \ge
	T_k
	\ge
	\sqrt{c_3}
	\int_0^1 \frac{\nu_3}{h_2(E - V(\gamma_k(s)))}\diff s.
	\end{equation*}
	Then, the sequence
	\[
	\int_{0}^{1}\frac{1}{E - V(\gamma_k(s))}\diff s
	\]
	is bounded.
	By \eqref{eq:bound-G} and \eqref{eq:bounded-lk},
	we have
	\begin{multline*}
	\int_{0}^{1}\norm{\dot{\gamma}_k(s)}^2\diff s
	\le 
	\int_{0}^{1}\frac{2 \nu_2}{E - V(\gamma_k(s))}
	G(\gamma_k(s),\dot{\gamma}_k(s))\diff s \\
	=
	\ell_k
	\int_{0}^{1}\frac{2 \nu_2}{E - V(\gamma_k(s))} \diff s 
	\le
	2\nu_2 c_2\int_{0}^{1}\frac{1}{E - V(\gamma_k(s))}\diff s,
	\end{multline*}
	so $(\gamma_k)$ is bounded in $W^{1,2}([0,1],\overline{D})$.
	By the Ascoli Arzelà Theorem,
	$\gamma_k$ uniformly converges to a curve $\gamma_y$,
	up to a subsequence.
	We claim that $\gamma_y$ is a minimizer for $\mathcal{J} $ in $X_y$.
	Let us show that $\gamma_y \in X_y$.
	Since $\gamma_k$ converges uniformly to $\gamma_y$, 
	$\gamma_y(0) = y $ and $\gamma(1) \in \partial D$.
	We show that $\gamma([0,1[)\subset D$ arguing by contradiction.
	Let $\bar{s} \in (0,1)$ be the first instant where
	$\gamma_y(\bar{s})\in \partial D$.
	By the minimality of $\gamma_k$,
	we have that $\gamma_y([\bar{s},1]) \subset \partial D$.
	Thus, we obtain 
	\[
	\lim_{k \to \infty} (1 - \bar{s})\ell_k 
	= \lim_{k \to \infty} \int_{\bar{s}}^1 G(\gamma_k(s),\dot{\gamma}_k(s))\diff s
	= \int_{\bar{s}}^1 G(\gamma_y(s),\dot{\gamma}_y(s))\diff s = 0,
	\]
	in contradiction with $\ell_k \ge c_1 > 0$,
	given by \eqref{eq:bounded-lk}.
	Hence, $\gamma_y$ belongs to $X_y$ and,
	since $\mathcal{J}(\gamma_y) \le \liminf_{k \to \infty} \ell_k$,
	by \eqref{eq:lkToPsiY} we obtain 
	$\mathcal{J}(\gamma_y ) = \psi(y)$.
	Being a minimizer,
	$\gamma_y$ satisfies \eqref{eq:cond-min-gammay}.
	By Lemma \ref{lem:fromGtoH},
	the diffeomorphism $\sigma\colon [0,T] \to [0,1]$ such that
	$\gamma_y \circ \sigma$
	is a solution of \eqref{eq:hamiltonianSys} with energy $E$ 
	has inverse
	\[
	t(s) = \sqrt{\mathcal{J}(\gamma_y)}
	\int_{0}^s
	\phi\left(\gamma_y(\sigma),
	\frac{\dot{\gamma_y}(\sigma)}{\sqrt{\mathcal{J}(\gamma_y)}}
	\right) \diff\sigma.
	\]
	By \eqref{eq:bounded-param},
	\[
	T = t(1) =  
	\sqrt{\mathcal{J}(\gamma_y)}
	\int_{0}^1
	\phi\left(\gamma_y(\sigma),\frac{\dot{\gamma}_y(\sigma)}{\sqrt{\mathcal{J}(\gamma_y)}}\right)
	\diff\sigma,
	\]
	is bounded and strictly greater than $0$.
\end{proof}

Recalling the definition of $\bar\delta$ given in Lemma \ref{lem:def-bardelta},
we give the following result.
\begin{proposition}
	\label{prop:psi-unique}
	For every $y \in D$ such that $\text{dist}(y,\partial D)\le \bar\delta$,
	the minimizer of $\mathcal{J}$ in the space $X_y$ is unique.
\end{proposition}
\begin{proof}
	By contradiction argument, let us assume the existence of two different curves,
	$\gamma_1,\gamma_2 \in X_y$, such that
	$\psi(y) = \mathcal{J}(\gamma_1) = \mathcal{J}(\gamma_2)$.
	Since $\gamma_1$ and $\gamma_2$ are two minimizers,
	by Proposition \ref{prop:psi-attains-min}		
	they are reparametrizations of two solutions of \eqref{eq:hamiltonianSys}
	with energy $E$ and final points on $\partial D$.
	We recall that, since $H$ is even with respect to $p$, 
	a backward parametrization of a solution is still a solution.
	Hence, if $\gamma_1(1) \ne \gamma_2(1)$, then
	setting $Q_1 = \gamma_1(1)$ and $Q_2 = \gamma_2(1)$, 
	we have
	\[
	y = q(t_1,Q_1)
	\quad \text{and} \quad
	y = q(t_2,Q_2),
	\]
	for some $t_1,t_2 > 0$.
	As a consequence,
	$y$ is given by two different coordinates of the 
	local chart constructed in Lemma \ref{lem:def-bardelta},
	which is a contradiction.
	If $\gamma_1(1) = \gamma_2(1)$,
	by the uniqueness of the solution $q(t,Q_y)$ stated in 
	Lemma \ref{lem:def-bardelta}
	we infer that $\gamma_1(s) = \gamma_2(s)$ for all $s \in [0,1]$,
	which is a contradiction.
\end{proof}
\begin{remark}
	\label{rmk:gamma_y-qtQ_y}
	By Proposition \ref{prop:psi-attains-min} and Proposition \ref{prop:psi-unique},
	for every $y \in D$ such that 
	$\text{dist}(y,\partial D)\le \bar\delta$,
	the minimizer $\gamma_y$ and the curve $q(t,Q_y)$
	are linked by a reparametrizations which inverts the orientation.
\end{remark}

\section{Differentiability and concavity}
Let $\bar\delta$ satisfies property \eqref{eq:def-bardelta},
and set
\begin{equation*}
\label{eq:def-Dbardelta}
D_{\bar\delta} = 
\left\{
y \in D: \text{dist}(y,\partial D)\le \bar\delta
\right\}.
\end{equation*}

\begin{proposition}
	\label{prop:psi-differentiable}
	For every $y \in D_{\bar\delta}$,
	$\psi$ is differentiable at $y$ and 
	\begin{equation}
	\label{eq:diff-psi}
	\diff\psi(y)[\xi] =
	- \frac{\partial G}{\partial v}(y,\dot{\gamma}_y(0))[\xi],
	\quad \forall \xi \in \mathbb{R}^{n}.
	\end{equation}
\end{proposition}
\begin{proof}
	Let us fix $\xi \in \mathbb{R}^n$, $\norm{\xi} \le 1$ and 
	set
	\[
	v_\xi(s) = \max\{0, 1 - 2s\}\xi,
	\]
	hence $v_{\xi}(s) = 0$ for all $ s \in [1/2,1]$.
	Let us define $\tilde{\mathcal{J}}\colon W^{1,2}([0,1],D) \to \mathbb{R}$ as
	\begin{equation*}
	\label{eq:def-tildeJ}
	\tilde{\mathcal{J}}(\gamma) = 
	\int_{0}^{\frac{1}{2}}G(\gamma(s),\dot{\gamma}(s))\diff s.
	\end{equation*}
	Since the curve $\gamma_y|_{[0,1/2]}$ is uniformly far from $\partial D$,
	so are the curves
	$(\gamma_y + \epsilon v_\xi)|_{[0,1/2]}$
	for $\epsilon$ sufficiently small.
	Moreover, 
	by Proposition \ref{prop:psi-attains-min},
	there exists a constant $c_\gamma > 0$ such that
	$G(\gamma_y(s),\dot{\gamma}_y(s)) = c_\gamma$, for all $s \in [0,1[$.
	As a consequence, for $\epsilon$ sufficiently small,
	we can assume that 
	\begin{equation}
	\label{eq:diff-psi-proof1}
	\dot{\gamma}_y(s) + \sigma\epsilon \dot{v}_{\xi}(s) \ne 0,
	\quad \forall s \in \left[0,\frac{1}{2}\right],\ 
	\forall \sigma \in [0,1],
	\end{equation}
	thus 
	we shall work in a region where $G$ is of class $C^2$.
	By definition of $\psi$ we have
	\[
	\psi(y + \epsilon\xi) \le \mathcal{J}(\gamma_y + \epsilon v_\xi).
	\]
	Since $\psi(y) = \mathcal{J}(\gamma_y)$, we obtain
	\[
	\psi(y + \epsilon\xi) - \psi(y)
	\le \mathcal{J}(\gamma_y + \epsilon v_\xi) - \mathcal{J}(\gamma_y) = 
	\tilde{\mathcal{J}}(\gamma_y + \epsilon v_\xi) -
	\tilde{\mathcal{J}}(\gamma_y),
	\]
	hence
	\[
	\limsup_{\epsilon \to 0} 
	\frac{1}{\epsilon}\bigg(
	\psi(y + \epsilon\xi) - \psi(y)\bigg)
	\le
	\limsup_{\epsilon \to 0} 
	\frac{1}{\epsilon}\bigg(
	\tilde{\mathcal{J}}(\gamma_y + \epsilon v_\xi) -
	\tilde{\mathcal{J}}(\gamma_y) \bigg).
	\]
	Then, using the dominated convergence theorem, an integration by parts
	and recalling that $\gamma_y$ satisfies 
	\eqref{eq:cond-min-gammay}, we have
	\begin{multline*}
	\lim_{\epsilon \to 0} 
	\frac{1}{\epsilon}\bigg(
	\tilde{\mathcal{J}}(\gamma_y + \epsilon v_\xi) -
	\tilde{\mathcal{J}}(\gamma_y)\bigg)\\
	= \lim_{\epsilon \to 0}  \frac{1}{\epsilon}
	\int_{0}^{\frac{1}{2}} \bigg(
	G(\gamma_y + \epsilon v_\xi, \dot{\gamma}_y + \epsilon \dot{v}_\xi) - 
	G(\gamma_y,\dot{\gamma}_y)
	\bigg)\diff s \\
	=
	\int_{0}^{\frac{1}{2}} 
	\left(
	\frac{\partial G}{\partial q}(\gamma_y,\dot{\gamma}_y)[v_{\xi}] 
	+ \frac{\partial G}{\partial v}(\gamma_y,\dot{\gamma}_y)[\dot{v}_\xi]
	\right)\diff s  \\
	= \left[\frac{\partial G}{\partial v}(\gamma_y,\dot{\gamma}_y)
	[\dot{v}_\xi]
	\right]^{1/2}_0
	= 
	- \frac{\partial G}{\partial v}(y,\dot{\gamma}_y(0))[\xi],
	\end{multline*}
	hence
	\begin{equation*}
	\label{eq:diff-psi-proof2}
	\limsup_{\epsilon \to 0} 
	\frac{1}{\epsilon}\bigg(
	\psi(y + \epsilon\xi) - \psi(y)\bigg)
	\le 
	- \frac{\partial G}{\partial v}(y,\dot{\gamma}_y(0))[\xi].
	\end{equation*}
	It remains to prove that
	\begin{equation}
	\label{eq:diff-psi-proof3}
	\liminf_{\epsilon \to 0} 
	\frac{1}{\epsilon}\bigg(
	\psi(y + \epsilon\xi) - \psi(y)\bigg) \ge
	- \frac{\partial G}{\partial v}(y,\dot{\gamma}_y(0))[\xi].
	\end{equation}
	Since $\psi(y + \epsilon\xi) = \mathcal{J}(\gamma_{y + \epsilon \xi})$
	and 
	$\psi(y) \le \mathcal{J}(\gamma_{y + \epsilon \xi} - \epsilon v_\xi)$,
	we have
	\begin{equation}
	\label{eq:diff-psi-proof5}
	\psi(y + \epsilon \xi) - \psi(y)
	\ge \mathcal{J}(\gamma_{y + \epsilon\xi}) -
	\mathcal{J}(\gamma_{y + \epsilon \xi} - \epsilon v_\xi) 
	= 
	\tilde{ \mathcal{J}}(\gamma_{y + \epsilon\xi}) -
	\tilde{\mathcal{J}}(\gamma_{y + \epsilon \xi} - \epsilon v_\xi).
	\end{equation}
	
	By \eqref{eq:diff-psi-proof1}, 
	$\tilde{\mathcal{J}} $ is of class $C^2$ in a neighbourhood of $\gamma_y$.
	Hence, there exists some $\sigma_\epsilon \in ]0,1[$ such that 
	\begin{equation}
	\label{eq:diff-psi-proof4}
	\tilde{\mathcal{J}}(\gamma_{y + \epsilon\xi}) -
	\tilde{\mathcal{J}}(\gamma_{y + \epsilon \xi} - \epsilon v_\xi) 
	= \epsilon\diff\tilde{\mathcal{J}}(\gamma_{y + \epsilon\xi})[v_\xi]
	-\frac{\epsilon^2}{2}
	\diff^2\tilde{\mathcal{J}}(\gamma_{y + \epsilon\xi} -
	\sigma_\epsilon \epsilon v_\xi)[v_\xi,v_\xi]
	.
	\end{equation}
	
	Now, we are going to prove that
	\begin{equation}
	\label{eq:diff-psi-proof9}
	\lim_{\epsilon \to 0}
	\epsilon
	\diff^2\tilde{\mathcal{J}}(\gamma_{y + \epsilon\xi} - \sigma_\epsilon \epsilon v_\xi)[v_\xi,v_\xi]
	= 0.
	\end{equation}
	Since $\gamma_y$ is uniformly far from $\partial D$ on the interval $[0,1/2]$,
	the same holds for $\gamma_{y + \epsilon\xi}$ whenever $\epsilon$ is sufficiently small.
	As a consequence, there exists a constant $c_1 > 0$ such that
	\[
	\frac{1}{E - V(\gamma_{y + \epsilon\xi}(s))} \le c_1,
	\quad \forall s \in \left[0,\frac{1}{2}\right].
	\]
	Since $\gamma_{y +\epsilon\xi}$ is a minimal geodesic, we also have
	\[
	G(\gamma_{y + \epsilon\xi}(s),\dot{\gamma}_{y + \epsilon\xi}(s)) =
	\psi(y + \epsilon \xi),
	\quad \forall s \in \left[0,\frac{1}{2}\right].
	\]
	Moreover,
	using \eqref{eq:bound-G},
	there exists a constant $c_2 > 0$ such that 
	\begin{multline}
	\label{eq:diff-psi-proof10}
	\int_{0}^{\frac{1}{2}}\norm{\dot{\gamma}_{y + \epsilon\xi}(s)}^2\diff s 
	\le 
	\int_{0}^{\frac{1}{2}}
	\frac{2\nu_2}{E - V(\gamma_{y + \epsilon\xi}(s))}
	G(\gamma_{y + \epsilon\xi}(s),\dot{\gamma}_{y + \epsilon\xi}(s)) \diff s
	\\
	\le c_1 \nu_2 \psi(y + \epsilon \xi)
	\le c_2.
	\end{multline}
	Hence, $\gamma_{y + \epsilon\xi}$ is uniformly bounded in 
	$W^{1,2}([0,\frac{1}{2}],D)$.
	Since $v_{\xi}(s) = 0$ on $[0,\frac{1}{2}]$,
	we have that 
	$
	\diff^2\tilde{\mathcal{J}}(\gamma_{y + \epsilon\xi} - \sigma_\epsilon \epsilon v_\xi)[v_\xi,v_\xi]
	$
	is uniformly bounded with respect to $\epsilon$ sufficiently small, 
	hence \eqref{eq:diff-psi-proof9} holds.
	By \eqref{eq:diff-psi-proof4} and \eqref{eq:diff-psi-proof9} we have
	\begin{equation}
	\label{eq:diff-psi-proof6}
	\lim_{\epsilon \to 0}\frac{1}{\epsilon}\bigg(
	\tilde{\mathcal{J}}(\gamma_{y + \epsilon\xi}) -
	\tilde{\mathcal{J}}(\gamma_{y + \epsilon \xi} - \epsilon v_\xi)\bigg) 
	=
	\lim_{\epsilon \to 0}
	\diff\tilde{\mathcal{J}}(\gamma_{y + \epsilon\xi})[v_{\xi}].
	\end{equation}
	Since $\gamma_{y + \epsilon\xi}$ satisfies \eqref{eq:cond-min-gammay},
	integration by parts leads to  
	\begin{equation}
	\label{eq:diff-psi-proof8}
	\diff\tilde{\mathcal{J}}(\gamma_{y + \epsilon\xi})[v_{\xi}] 
	= -\frac{\partial G}{\partial v}(y + \epsilon\xi,\dot{\gamma}_{y + \epsilon\xi}(0))
	[\xi].
	\end{equation}
	To obtain \eqref{eq:diff-psi-proof3} and conclude the proof,
	it suffices to show that 
	\begin{equation}
	\label{eq:diff-psi-proof99}
	\lim_{\epsilon \to 0} \dot{\gamma}_{y + \epsilon \xi}(0) = 
	\dot{\gamma}_y(0).
	\end{equation}
	To this aim, we exploit the uniqueness of $\gamma_y$
	ensured by Proposition \ref{prop:psi-unique}.
	Arguing by contradiction, 
	let 
	$(\epsilon_n)$ be a sequence such that 
	$\epsilon_n \to 0$ and 
	\[
	\lim_{n \to \infty} \dot{\gamma}_{y + \epsilon_n \xi}(0) \ne 
	\dot{\gamma}_y(0).
	\]
	By \eqref{eq:diff-psi-proof10}, 
	$\gamma_{y + \epsilon \xi}$ are uniformly bounded in $W^{1,2}([0,1],\overline{D})$,
	hence there exists $v \in \mathbb{R}^n$ such that
	$\lim_{n \to \infty} \dot{\gamma}_{y + \epsilon_n \xi}(0) = v \ne \dot{\gamma}_y(0)$.
	Since $(\gamma_{y + \epsilon_n \xi})$ is a sequence of geodesics, 
	it converges with respect to the $C^1$ norm
	to a minimum.
	Since the minimum is unique by Proposition \ref{prop:psi-unique},
	then $\gamma_{y + \epsilon_n \xi}\to \gamma_y$ in $C^1$,
	so
	$\dot{\gamma}_{y + \epsilon_n \xi}(0) \to \dot{\gamma}_y(0)$,
	which is a contradiction.
	
	Therefore,  \eqref{eq:diff-psi-proof99} holds and using also
	\eqref{eq:diff-psi-proof8} we have
	\begin{equation}
	\label{eq:diff-psi-proof7}
	\lim_{\epsilon \to 0}
	\diff\mathcal{J}(\gamma_{y + \epsilon\xi})[v_{\xi}]
	= - \frac{\partial G}{\partial v}(y,\dot{\gamma}_y(0))[\xi].
	\end{equation}
	Finally, combining
	\eqref{eq:diff-psi-proof5},
	\eqref{eq:diff-psi-proof6} and
	\eqref{eq:diff-psi-proof7},
	we obtain \eqref{eq:diff-psi-proof3}
	and we are done.
\end{proof}

The next lemma will play a central role 
because it links the initial velocity of the curve
$\gamma_y$ with $\dot{q}(t_y,Q_y)$ through a function of class $C^1$.
\begin{lemma}
	\label{lem:def-varphi}
	There exists a function
	$\varphi\colon D_{\bar\delta}\times \mathbb{R}^n \to \mathbb{R}$
	of class $C^1$ such that
	\begin{equation}
	\label{eq:from-qy-to-gammay}
	\sqrt{\psi(y)}\  \varphi(y,\dot{q}(t_y,Q_y))\dot{q}(t_y,Q_y) = -\dot{\gamma}_y(0).
	\end{equation}
\end{lemma}
\begin{proof}
	Let $\zeta_y\colon [0,\sqrt{\psi(y)}] \to \overline{D}$
	the backward arc-length reparametrization of $\gamma_y$,
	namely
	\[
	\zeta_y(s) = \gamma_y\left(1 - \frac{s}{\sqrt{\psi(y)}}\right).
	\]
	As a consequence,
	\begin{equation}
	\label{eq:varphi-proof1}
	\dot{\zeta}_y(\sqrt{\psi(y)})
	=- \frac{1}{\sqrt{\psi(y)}}\dot{\gamma}_y(0).
	\end{equation}
	By Lemma \ref{lem:def-G},
	the curve $x\colon [0,\sqrt{\psi(y)}] \to \mathbb{R}^{2n}$ given by
	\[
	x(s) = \mathcal{L}^{-1}\left(\zeta_y(s),\dot{\zeta}_y(s)\right),
	\]
	is a solution of Hamilton's equations with respect to $U$
	and $U(x(s))\equiv 1$.
	Since $y \in D_{\bar\delta} $,
	By Lemma \ref{lem:def-bardelta} and
	Remark \ref{rmk:gamma_y-qtQ_y},
	$x(s)$ is actually a reparametrization of 
	$z(t,Q_y) = (q(t,Q_y),p(t,Q_y))$, with 
	$x(\sqrt{\psi(y)}) = z(t_y,Q_y)$.
	Hence, using \eqref{eq:fromZtoX} and
	recalling that $\mathcal{L}$ is the identity map
	with respect to the first variable, we have
	\begin{multline}
	\label{eq:varphi-proof2}
	\dot{\zeta}_y(\sqrt{\psi(y)}) =
	\frac{\norm{U'(z(t_y,Q_y))}}{\norm{H'(z(t_y,Q_y))}}
	\dot{q}(t_y,Q_y) \\
	=
	\frac{\norm{U'(y,p(t_y,Q_y))}}{\norm{H'(y,p(t_y,Q_y))}}
	\dot{q}(t_y,Q_y).
	\end{multline}
	Since the map
	\[
	p(t_y,Q_y) \mapsto \dot{q}(t_y,Q_y)
	= \frac{\partial H}{\partial p}\left(y,p(t_y,Q_y)\right)
	\]
	is invertible, there exists 
	$\varphi\colon D_{\bar\delta}\times \mathbb{R}^n\to \mathbb{R}$
	such that 
	\begin{equation}
	\label{eq:varphi-proof3}
	\varphi(y,\dot{q}(t_y,Q_y))
	= 
	\frac{\norm{U'(y,p(t_y,Q_y))}}{\norm{H'(y,p(t_y,Q_y))}}.
	\end{equation}
	Combining \eqref{eq:varphi-proof1},
	\eqref{eq:varphi-proof2} and \eqref{eq:varphi-proof3},
	we obtain \eqref{eq:from-qy-to-gammay}.
	Recalling that both $t_y$ and $Q_y$ are 
	of class $C^1$ by Remark \ref{rmk:tyQyC1},
	and that
	$\dot{q}(t_y,Q_y) \ne 0$
	for every $y \in D_{\bar\delta}$,
	the function $\varphi$ is of class $C^1$
	as a composition of the derivatives of $H$ and $U$,
\end{proof}

\begin{lemma}
	\label{lem:psiC2}
	The function $\psi$ is of class $C^2$ in $D_{\bar\delta}$.
\end{lemma}
\begin{proof}
	By \eqref{eq:from-qy-to-gammay}, we deduce that $\dot{\gamma}_y(0)$ is continuous
	as a function of $y \in D_{\bar\delta}$.
	Hence, by \eqref{eq:diff-psi}, $\psi$ is of class $C^1$ in $D_{\bar\delta}$.
	Using again \eqref{eq:from-qy-to-gammay} and
	the $C^1$-regularity of $\varphi$ and $\dot{q}(t_y,Q_y)$,
	we deduce that $\dot{\gamma}_y(0)$ is of class $C^1$.
	Using \eqref{eq:diff-psi}, we obtain the thesis.
\end{proof}

Recalling the notion of concavity given in 
Definition \ref{def:Finsler-concavity}
and the definition of $H_{\psi}(y,v)[v,v]$ in \eqref{eq:def-Concave},
the next proposition shows that the set 
$\psi([\hat\delta,\infty[)$ is concave,
provided $\hat\delta > 0$ sufficiently small.

\begin{proposition}
	\label{prop:concavity}
	There exists $\hat\delta \in ]0,\bar\delta]$ such that,
	for every $y \in D$ with
	$0 < \text{dist}(y,\partial D) \le \hat\delta$,
	\begin{equation*}
	\label{eq:concavity-psi}
	H_{\psi}(y,\xi)[\xi,\xi] > 0,
	\quad \forall \xi \in \mathbb{R}^n\backslash\left\{0\right\}:
	\diff\psi(y)[\xi] = 0.
	\end{equation*}
\end{proposition}
\begin{proof}
	For every $\xi \in \mathbb{R}^n\backslash\{0\}$
	such that $\diff\psi(y)[\xi]=0$,
	we denote by $\eta$ the unique Finsler geodesic such that
	$\eta(0) = y$ and $\dot{\eta}(0) = \xi$.
	We have to prove that, for $y$ sufficiently near the boundary $\partial D$,
	\[
	\frac{\diff^2}{\diff s^2}(\psi\circ\eta)(0) > 0.
	\]
	Let $\zeta$ be the reparametrization of $\eta$
	which is a solution of \eqref{eq:hamiltonianSys}
	with energy $E$.
	By Remark \ref{rmk:fromUtoH},
	there exists a function $\lambda$ of class $C^2$ such that
	$\zeta(t) = \eta(\lambda(t))$, $\lambda(0) = 0$ and $\dot{\lambda}(t) > 0$.
	Hence,
	\[
	\frac{\diff^2}{\diff s^2}(\psi\circ\zeta)(0)  = 
	\dot{\lambda}(0) \frac{\diff^2}{\diff s^2}(\psi\circ\eta)(0)+
	\ddot{\lambda}(0)\diff\psi(y)[\xi] 
	=
	\dot{\lambda}(0)	\frac{\diff^2}{\diff s^2}(\psi\circ\eta)(0).
	\]
	As a consequence, it suffices to prove that 
	\[
	\frac{\diff^2}{\diff s^2}(\psi \circ \zeta)(0)> 0.
	\]
	Using \eqref{eq:diff-psi} and \eqref{eq:from-qy-to-gammay}  we obtain
	\[
	\begin{multlined}
	\frac{\diff}{\diff s}\psi(\zeta(s))=
	\diff\psi(\zeta(s))[\dot{\zeta}(s)]=
	-\frac{\partial G}{\partial v}\left(\zeta(s),\dot{\gamma}_{\zeta(s)}(0)\right)[\dot{\zeta}(s)]\\
	= \sqrt{\psi(\zeta(s))}\varphi\left(\zeta(s),\dot{q}(t_{\zeta(s)},Q_{\zeta(s)})\right)
	\frac{\partial G}{\partial v}(\zeta(s),\dot{q}(t_{\zeta(s)},Q_{\zeta(s)}))[\dot{\zeta}(s)].
	\end{multlined}
	\]
	Let us set $w = \dot{\zeta}(0)$.
	Since $\dot{q}(t_y,Q_y)$ is parallel to $\dot\gamma_y(0)$, we have
	\[
	\frac{\partial G}{\partial v}(y,\dot{q}(t_y,Q_y))[w] = 0,
	\]
	thus we obtain
	\begin{multline*}
	\frac{\diff^2}{\diff s^2}(\psi\circ\zeta)(0) \\
	=
	\frac{\diff}{\diff s}
	\left(\sqrt{\psi(\zeta(s))}\varphi(\zeta(s),\dot{q}(t_{\zeta(s)},Q_{\zeta(s)}))\right)\bigg|_{s = 0}
	\frac{\partial G}{\partial v}(y,\dot{q}(t_y,Q_y))[w] \\
	+
	\sqrt{\psi(y)} \varphi(y,\dot{q}(t_y,Q_y))
	\frac{\diff }{\diff s}
	\left(
	\frac{\partial G}{\partial v}(\zeta(s),\dot{q}(t_{\zeta(s)},Q_{\zeta(s)}))[\dot{\zeta}(s)]
	\right)\bigg|_{s = 0}\\
	=
	\sqrt{\psi(y)} \varphi(y,\dot{q}(t_y,Q_y))
	\frac{\diff }{\diff s}
	\left(
	\frac{\partial G}{\partial v}(\zeta(s),\dot{q}(t_{\zeta(s)},Q_{\zeta(s)}))[\dot{\zeta}(s)]
	\right)\bigg|_{s = 0}.
	\end{multline*}
	Since $\psi(y) $ and $\varphi(y,\dot{q}(t_y,Q_y))$
	are two strictly positive functions with respect to $y \in D$,
	it remains to prove that, for $y$ sufficiently near the boundary $\partial D$, we have
	\begin{equation}
	\label{eq:concavity-Proof1}
	\frac{\diff }{\diff s}
	\left(
	\frac{\partial G}{\partial v}(\zeta(s),\dot{q}(t_{\zeta(s)},Q_{\zeta(s)}))[\dot{\zeta}(s)]
	\right)\bigg|_{s = 0} > 0.
	\end{equation}
	Let us define
	the function
	$\Gamma\colon D\times \mathbb{R}^n\to \mathbb{R}$ as
	\begin{equation*}
	\label{eq:def-Gamma}
	\Gamma(q,v) = \frac{G(q,v)}{ E - V(q)}.
	\end{equation*}
	Since $\diff\psi(y)[w] = 0$,
	we have 
	\[
	\frac{\partial \Gamma}{\partial v}(y,\dot{q}(t_y,Q_y))[w] = 
	\frac{1}{E - V(y)}
	\frac{\partial G}{\partial v}(y,\dot{q}(t_y,Q_y))[w] = 0.
	\]
	As a consequence,
	\begin{equation*}
	\begin{multlined}
	\frac{\diff }{\diff s}
	\left(
	\frac{\partial G}{\partial v}(\zeta(s),\dot{q}(t_{\zeta(s)},Q_{\zeta(s)}))[\dot{\zeta}(s)]
	\right)\bigg|_{s = 0} \\
	= 
	\left(E - V(y)\right)
	\frac{\diff}{\diff s}
	\left(
	\frac{\partial \Gamma}{\partial v}(\zeta(s),\dot{q}(t_{\zeta(s)},Q_{\zeta(s)}))[\dot{\zeta}(s)]
	\right)\bigg|_{s = 0}.
	\end{multlined}
	\end{equation*}
	Hence, to obtain \eqref{eq:concavity-Proof1}, it suffices to prove that
	\begin{equation}
	\label{eq:concavity-Proof2}
	\frac{\diff }{\diff s}
	\left(
	\frac{\partial \Gamma}{\partial v}(\zeta(s),\dot{q}(t_{\zeta(s)},Q_{\zeta(s)}))[\dot{\zeta}(s)]
	\right)\bigg|_{s = 0} > 0,
	\end{equation}
	for $y$ sufficiently near the boundary $\partial D$.
	Setting
	\begin{align*}
	I_1(y) &=
	\frac{\diff }{\diff s}
	\left(
	\frac{\partial \Gamma}{\partial v}(\zeta(s),\dot{q}(t_{\zeta(s)},Q_{\zeta(s)}))
	\right)\bigg|_{s = 0} [w],\\
	I_2(y) &= \frac{\partial \Gamma}{\partial v}(y,\dot{q}(t_y,Q_y))[\ddot{\zeta}(0)],
	\end{align*}
	we have
	\[
	\frac{\diff }{\diff s}
	\left(
	\frac{\partial \Gamma}{\partial v}(\zeta(s),\dot{q}(t_{\zeta(s)},Q_{\zeta(s)}))[\dot{\zeta}(s)]
	\right)\bigg|_{s = 0} = I_1(y) + I_2(y).
	\]
	Let us study $I_1(y)$ and $I_2(y)$ separately.
	We have
	\begin{multline*}
	I_1(y) =
	\frac{\partial^2 \Gamma}{\partial q \partial v}
	\left(y,\dot{q}(t_y,Q_y)\right)[w,w] \\
	+ \frac{\partial^2 \Gamma}{\partial v^2}\left(y,\dot{q}(t_y,Q_y)\right)
	\left[
	\frac{\diff }{\diff s}\left(\dot{q}(t_{\zeta(s)},Q_{\zeta(s)})\right)
	\bigg|_{s = 0},w\right].
	\end{multline*}
	Since $q(t,Q_y)$ is a solution of \eqref{eq:hamiltonianSys}
	with energy $E$, by \eqref{eq:bound-deriv-K-nu}
	we have
	\[
	\nu_1 \norm{p(t_y,Q_y)}\le
	\norm*{\frac{\partial K}{\partial p}(y,p(t_y,Q_y))}
	= \norm{\dot{q}(t_y,Q_y)}
	\le
	\nu_2 \norm{p(t_y,Q_y)},
	\]
	and by \eqref{eq:bound-K-nu}
	we obtain that there exist two constants $c_1,c_2 > 0$ such that 
	\begin{equation}
	\label{eq:concavity-Proof-3a}
	c_1(E - V(y))\le \norm{\dot{q}(t_y,Q_y)}^2 \le c_2 (E - V(y)).
	\end{equation}
	Similarly, since $w = \dot{\zeta}(0)$, we have
	\begin{equation}
	\label{eq:concavity-Proof-3}
	c_1(E - V(y))\le \norm{w}^2 \le c_2 (E - V(y)).
	\end{equation}
	Since $\Gamma$ is homogeneous of degree two with respect to $v$
	and recalling the bounds for $G$ given by \eqref{eq:bound-G},
	there exists a constant $c_3$ such that
	\[
	\norm*{\frac{\partial^2 \Gamma}{\partial q \partial v}(q,v)[\omega,\omega]}\le
	c_3 \norm{v}\norm{\omega}^2,
	\quad \forall (q,v) \in D \times \mathbb{R}^n.
	\]
	As a consequence, using
	\eqref{eq:concavity-Proof-3a} and
	\eqref{eq:concavity-Proof-3}
	we have
	\begin{equation}
	\label{eq:concavity-proof4}
	\norm*{
		\frac{\partial^2 \Gamma}{\partial q \partial v}
		\left(y,\dot{q}(t_y,Q_y)\right)[w,w]
	} \le c_3 (E-V(y))^{\frac{3}{2}}.
	\end{equation}
	Let us set
	\[
	v_y \coloneqq 
	\frac{\partial^2 H}{\partial p^2}(y,0) \frac{\partial V}{\partial q}(y).
	\]
	We remark that,
	since $(\partial V/\partial q)(q) \ne 0$ for every $ q \in \partial D$,
	and by the strictly convexity of $H$ given by \eqref{eq:bound-Hessian-K-nu},
	$\norm{v_y} $ is uniformly greater than zero 
	if $y $ is sufficiently near the boundary. 
	Since $\lim_{s \to 0}Q_{\zeta(s)} = Q_y$, 
	using \eqref{eq:dotq-nearBoundary} we have
	\begin{equation*}
	\label{eq:concavity-dotq}
	\dot{q}(t,Q_{\zeta(s)}) = - t v_y + \rho(t,Q_{\zeta(s)}),
	\quad \forall t \in [0, +\infty[,
	\end{equation*}
	with $\diff \rho(0,Q_y) = 0$.
	By Remark \ref{rmk:tyQyC1}, 
	$t_y$ and $Q_y$ are functions of class $C^1$ with respect to $y$.
	Moreover, by Remark \ref{rmk:dzC1}, $\dot{q}(t,Q)$ is of class $C^1$.
	Hence, we obtain
	\begin{multline}
	\label{eq:concavity-diff-dotq}
	\frac{\diff}{\diff s}
	\left(\dot{q}(t_{\zeta(s)},Q_{\zeta(s)})\right)\bigg|_{s = 0} 
	= - \diff t_y[w] v_y + 
	\frac{\partial \rho}{\partial t}(t_y,Q_y) \diff t_y[w] \\
	+ \frac{\partial \rho}{\partial Q}(t_y,Q_y)
	\frac{\partial Q}{\partial y}[w].
	\end{multline}
	Since $q(t_y,Q_y) = y$, we get
	\begin{equation}
	\label{eq:concavity-proof1a}
	\diff t_y[v]\dot{q}(t_y,Q_y)
	+ \frac{\partial q}{\partial Q}(t_y,Q_y)\frac{\partial Q_y}{\partial y}[v] = v,
	\qquad \forall v \in \mathbb{R}^n.
	\end{equation}
	We recall that $(t,Q)$ is a coordinate system
	in a neighbourhood of $\partial D$, where
	$y = q(t_y,Q_y)$.
	Hence,
	if $y$ tends to $\partial D$,
	then $t_y \to 0$ and
	$(\partial q/\partial Q)(t_y,Q_y)$ goes to the identity map.
	Similarly, when $y\to \partial D$,
	$(\partial Q_y/\partial y)[v]$ tends to $v$ uniformly as $\norm{v}\le 1$.
	Then,
	by \eqref{eq:concavity-proof1a},
	$\diff t_y[v]\dot{q}(t_y,Q_y)\to 0$ uniformly in $v$,
	as $y \to \partial D$.
	Therefore,
	since by \eqref{eq:concavity-Proof-3a} and \eqref{eq:concavity-Proof-3}
	$w$ and $\dot{q}(t_y,Q_y)$
	we have 
	\[
	0 <
	\sqrt{\frac{c_1}{c_2}}
	\le \frac{\norm{\dot{q}(t_y,Q_y)}}{\norm{w}}
	\sqrt{\frac{c_2}{c_1}},
	\]
	we obtain
	\begin{equation}
	\label{eq:concavity-diff-ty}
	\lim_{y \to \partial D} \diff t_y[w] =  0.
	\end{equation}
	Since $\diff \rho(0,Q_y) = 0$,
	by \eqref{eq:concavity-diff-dotq} and \eqref{eq:concavity-diff-ty}
	we infer
	\begin{equation}
	\label{eq:concavity-diff-dotq-o1}
	\frac{\diff}{\diff s}\left(\dot{q}(t_{\zeta(s)},Q_{\zeta(s)})\right)\bigg|_{s = 0}
	= o(1),
	\end{equation}
	as $y \to \partial D$.
	Since $\Gamma$
	is homogeneous of degree 2 with respect to $v$
	and using \eqref{eq:bound-G},
	there exist a constant $c_3 > 0 $ such that 
	\[
	\norm*{\frac{\partial^2 \Gamma}{\partial v^2}(q,\xi)[v_1,v_2]}
	\le c_4\norm{v_1}\norm{v_2},
	\quad \forall q \in D,\ \forall \xi \in \mathbb{R}^n\backslash \{0\},
	\ \forall v_1,v_2 \in \mathbb{R}^n.
	\]
	As a consequence,
	by \eqref{eq:concavity-Proof-3} and \eqref{eq:concavity-diff-dotq-o1}
	we obtain
	\begin{multline}
	\label{eq:concavity-proof5}
	\lim_{y \to \partial D}
	\frac{\partial^2 \Gamma}{\partial v^2}\left(y,\dot{q}(t_y,Q_y)\right)
	\left[ \frac{\diff }{\diff s}\left(\dot{q}(t_{\zeta(s)},Q_{\zeta(s)})\right)
	\bigg|_{s = 0},w\right] \\
	= o\left(\sqrt{E - V(y)}\right).
	\end{multline}
	By \eqref{eq:concavity-proof4} and \eqref{eq:concavity-proof5} we obtain
	\begin{equation}
	\label{eq:concavity-proof6}
	I_1(y) = o\left(\sqrt{E - V(y)}\right),
	\end{equation}
	as $y \to \partial D$.
	
	Let us analyse $I_2(y)$.
	Since $(\zeta,p)$ is a solution of Hamilton's equations,
	with $p$ implicitly defined by
	\[
	\dot{\zeta}(s) = \frac{\partial H}{\partial p}(\zeta(s),p(s)),
	\]
	we have
	\[
	\ddot{\zeta}(0) = 
	\frac{\partial^2 H}{\partial q\partial p}(y,p(0)) w 
	- \frac{\partial^2 H}{\partial p^2}(y,p(0))
	\left(
	\frac{\partial K}{\partial q}(y,p(0))
	+ \frac{\partial V}{\partial q}(y)
	\right).
	\]
	If $y \to \partial D$, then $\norm{p(0)} \to 0$.
	Therefore,
	\[
	\frac{\partial^2 H}{\partial q\partial p}(y,p(0)) w \to 0,
	\qquad
	\frac{\partial K}{\partial q}(y,p(0)) \to 0,
	\]
	and we obtain
	$
	\lim_{y \to \partial D} \ddot{\zeta}(0) = 
	- v_y.
	$
	Hence, 
	using also
	\eqref{eq:bound-G} and
	\eqref{eq:concavity-diff-dotq},
	we obtain
	\begin{multline}
	\label{eq:concavity-proof7}
	\lim_{y \to \partial D}
	\frac{\partial \Gamma}{\partial v}
	\left(y,\frac{\dot{q}(t_y,Q_y)}{\norm{\dot{q}(t_y,Q_y)}}\right)
	\left[\ddot{\zeta}(0)\right]= 
	\frac{\partial \Gamma}{\partial v}(y, - v_y)[-v_y]\\
	=
	\Gamma\left( y, v_y \right)
	\ge
	\frac{1}{2\nu_2}
	\norm{v_y}^2.
	\end{multline}
	As a consequence, by 
	\eqref{eq:concavity-Proof-3a} and
	\eqref{eq:concavity-proof7} we obtain
	\begin{equation}
	\label{eq:concavity-proof8}
	\lim_{y \to \partial D}
	\frac{I_2(y)}{\sqrt{E - V(y)}}
	> 0.
	\end{equation}
	Finally, by \eqref{eq:concavity-proof6} and \eqref{eq:concavity-proof8},
	we obtain \eqref{eq:concavity-Proof2} and we are done.
\end{proof}

\section{Proof of the main theorem}%
\label{sec:proof_of_the_main_theorem}

Finally, we are ready to prove Theorem \ref{teo:main}.
\begin{proof}[Proof of Theorem \ref{teo:main}]
	Let $\hat\delta$ be as in Proposition \ref{prop:concavity}
	and set
	\[
	\Omega = \psi^{-1}(]\hat\delta,+\infty[).
	\]
	By continuity, $\Omega$ is an open subset of $D$ and $\partial\Omega =
	\psi^{-1}(\hat\delta)$.
	By Lemma \ref{lem:psiC2}, $\psi$ is of class $C^2 $ in $D_{\bar\delta}$, 
	and since $\diff\psi \ne 0$ on $\partial\Omega$,
	we have that $\partial\Omega$ is of class $C^2$.
	Since $\bar\delta$ satisfies property \eqref{eq:def-bardelta} and 
	$\hat\delta \le \bar\delta$, $\overline{\Omega}$ is homeomorphic to $\overline{D}$.
	Since $\partial\Omega$ is a level hyper-surface of $\psi$,
	for every $y \in \partial\Omega$, $v \in T_y\partial\Omega$ if and only if $\diff\psi(y)[v] = 0$.
	Recalling Definition \ref{def:Finsler-concavity},
	Proposition \ref{prop:concavity} implies that
	$\overline{\Omega}$ is strictly concave with respect to the Finsler metric $F$. 
	
	Let $\gamma\colon [0,1] \to \overline{\Omega}$ be an orthogonal Finsler geodesic chord.
	We will prove the desired properties of the extension $\hat\gamma\colon [\alpha,\beta] \to \overline{D}$
	only in the interval $[1,\beta]$.
	The case $[\alpha,0]$ is analogue.
	Set $y = \gamma(1)$.
	Since $\gamma$ is an orthogonal Finsler geodesic chord,
	it satisfies \eqref{eq:conBounCond},
	hence 
	\[
	\frac{\partial G}{\partial v}(y,\dot{\gamma}(1))[v] = 0,
	\quad \forall v \in T_{\gamma(1)}\partial\Omega.
	\]
	The minimizer curve $\gamma_y$ satisfies
	\[
	\frac{\partial G}{\partial v}(y,\dot{\gamma}_y(0))[v] = 0,
	\quad \forall v \in T_{\gamma(1)}\partial\Omega,
	\]
	thus $\dot{\gamma}(1)$ and $\dot{\gamma}_y(0)$ are parallel.
	As a consequence, the curve 
	$	\bar\gamma\colon [0,2] \to \overline{D}	$ defined as
	\[
	\bar\gamma(s) = 
	\begin{cases}
	\gamma(s), &		\mbox{if } s \in [0,1], \\
	\gamma_y(s - 1), &		\mbox{if } s \in ]1,2], 
	\end{cases}
	\]
	is of class $C^1$	and it is a geodesic with respect to $F$,
	up to a suitable time reparametrization.
	With the analogue extension in $[\alpha,0]$, 
	we obtain a geodesic $\hat\gamma\colon[\alpha,\beta]\to\overline{D}$
	such that $\hat\gamma(\alpha),\hat\gamma(\beta) \in \partial D$ 
	and $\hat{\gamma}(]\alpha,\beta[) \subset D$.
	By Lemma \ref{lem:timeParam-HtoU} and Lemma \ref{lem:def-G},
	we have that
	\[
	(q(t),p(t)) =\mathcal{L}^{-1}(\hat\gamma(t),\dot{\hat\gamma}(t)) 
	\quad	\forall t \in ]\alpha,\beta[,
	\]
	is a solution of \eqref{eq:hamiltonianSys} with energy $E$ for $H$, 
	up to time reparametrization.
	Using also Lemma \ref{lem:finiteTk} to ensure that 
	the time reparametrization is finite, 
	we obtain the existence of a diffeomorphism $\sigma\colon [0,T] \to [\alpha,\beta]$, 
	with $\sigma(0) = \alpha$ and $\sigma(T) = \beta$,
	such that
	\[
	(q,p)\circ \sigma\colon [0,T] \to \overline{\Sigma}
	\]
	is a brake orbit.
\end{proof}

\end{document}